\documentclass[a4paper,12pt,leqno]{article}

\usepackage[utf8]{inputenc}
\usepackage[T1]{fontenc}
\usepackage{lmodern}
\usepackage{amsmath}
\usepackage{amsthm}
\usepackage{amssymb}
\usepackage{mathrsfs}
\usepackage{hyperref}

\newcommand{\N}{\mathbb N}
\newcommand{\Z}{\mathbb Z}
\newcommand{\Q}{\mathbb Q}
\newcommand{\R}{\mathbb R}
\newcommand{\C}{\mathbb C}
\newcommand{\T}{\mathbb T}
\newcommand{\D}{\mathbb D}

\newcommand{\F}{\mathbb F}

\newcommand{\cL}{\mathcal L}

\newcommand{\Hil}{\mathscr H}

\newcommand{\support}{\operatorname{supp}}
\newcommand{\trace}{\operatorname{Tr}}

\newcommand{\sspan}{\ensuremath{\operatorname{span}}}
\newcommand{\linbeg}{\boldsymbol{\operatorname{B}}}

\newcommand{\traceclass}{\linbeg^1}

\newcommand{\cont}{C}

\newcommand{\cstar}{\ensuremath{C^*}}
\newcommand{\cb}{\mathrm{cb}}

\newcommand{\cpt}{\mathrm c}
\newcommand{\unit}{\boldsymbol 1}
\newcommand{\dd}{\mathrm d}

\newcommand{\GL}{{GL}}

\newcommand{\PGL}{{PGL}}

\newcommand{\ELL}{{L}}

\newcommand{\ip}[2]{\langle#1,#2\rangle}

\newcommand{\seto}[1]{\seta{#1}}
\newcommand{\seta}[1]{\ensuremath{\{#1\}}}
\newcommand{\setw}[2]{\setaw{#1}{#2}}
\newcommand{\setaw}[2]{\seta{#1\::\:#2}}
\newcommand{\Email}{\begingroup \def\UrlLeft{<}\def\UrlRight{>} \urlstyle{tt}\Url}
\newcommand{\mailto}[1]{\href{mailto:#1}{\Email{#1}}}

\newcommand{\contrib}[3]{#1\quad\mailto{#2}{\small\\\quad\textit{#3}}\\[1ex]}

\newcommand{\FA}{A}

\newcommand{\FSA}{B}

\newcommand{\kk}{c}
\newcommand{\eval}{s}
\newcommand{\FF}{(\ast_{m=1}^M\Z/2\Z)\ast(\ast_{n=1}^N\Z)}

\newcommand{\cc}{c}
\newcommand{\G}{G}
\newcommand{\K}{K}
\newcommand{\X}{X}
\newcommand{\Y}{Y}
\newcommand{\GLQ}{\GL_2(\padicnum)}
\newcommand{\PGLQ}{\PGL_2(\padicnum)}
\newcommand{\SLZ}{\GL_2(\padicint)}
\newcommand{\PSLZ}{\PGL_2(\padicint)}

\newcommand{\padicnum}{\Q_q}
\newcommand{\padicnums}{\Q_q^*}
\newcommand{\padicint}{\Z_q}
\newcommand{\padicints}{\Z_q^*}

\newcommand{\MA}{MA}
\newcommand{\MoA}{M_0A}

\newcommand{\PG}{P_\G}
\newcommand{\QG}{Q_\G}
\newcommand{\PX}{P_\X}
\newcommand{\QX}{Q_\X}

\swapnumbers
\numberwithin{equation}{section}
\theoremstyle{plain}
\newtheorem{theorem}{Theorem}[section]
\newtheorem{maintheorem}[theorem]{Theorem}
\newtheorem{proposition}[theorem]{Proposition}
\newtheorem{lemma}[theorem]{Lemma}
\newtheorem{corollary}[theorem]{Corollary}

\theoremstyle{definition}
\newtheorem{definition}[theorem]{Definition}

\theoremstyle{remark}
\newtheorem{remark}[theorem]{Remark}

\renewcommand{\vec}{\boldsymbol}
\renewcommand{\Re}{\mathrm{Re}}
\renewcommand{\Im}{\mathrm{Im}}
\renewcommand{\contrib}[2]{#1\quad\mailto{#2}}

\begin{document}
	\hyphenation{hil-bert-space fran-ces-co lo-rentz boun-ded func-tions haa-ge-rup pro-per-ties a-me-na-bi-li-ty re-wri-ting pro-per-ty tro-els steen-strup jen-sen de-note as-su-ming Przy-by-szew-ska ma-jo-ri-za-tion o-pe-ra-tors mul-ti-pli-ers i-so-mor-phism la-place bel-tra-mi re-pre-sen-ta-tions pro-po-si-tion the-o-rem lem-ma co-rol-la-ry}

	\title{Schur Multipliers and Spherical Functions on Homogeneous Trees}
\author{Uffe Haagerup\thanks{Partially supported by the Danish Natural Science Research Council.} \and Troels Steenstrup\thanks{Partially supported by the Ph.D.-school OP--ALG--TOP--GEO.} \and Ryszard Szwarc\thanks{Supported by European Commission Marie Curie Host Fellowship for the Transfer of Knowledge ``Harmonic Analysis, Nonlinear Analysis and Probability''  MTKD-CT-2004-013389 and by MNiSW Grant N201 054 32/4285.}}
\date{\today}
\maketitle

	\begin{abstract}
		Let $\X$ be a homogeneous tree of degree $q+1$ ($2\leq q\leq\infty$) and let $\psi:\X\times\X\to\C$ be a function for which $\psi(x,y)$ only depend on the distance between $x,y\in\X$. Our main result gives a necessary and sufficient condition for such a function to be a Schur multiplier on $\X\times\X$. Moreover, we find a closed expression for the Schur norm $\|\psi\|_S$ of $\psi$. As applications, we obtain a closed expression for the completely bounded Fourier multiplier norm $\|\cdot\|_{\MoA(\G)}$ of the radial functions on the free (non-abelian) group $\F_N$ on $N$ generators ($2\leq N\leq\infty$) and of the spherical functions on the p-adic group $\PGLQ$ for every prime number $q$.

	\end{abstract}
	\section*{Introduction}
\label{intro}
Let $\Y$ be a non-empty set. A function $\psi:\Y\times\Y\to\C$ is called a \emph{Schur multiplier} if for every operator $A=(a_{x,y})_{x,y\in\Y}\in\linbeg(\ell^2(\Y))$ the matrix $(\psi(x,y)a_{x,y})_{x,y\in\Y}$ again represents an operator from $\linbeg(\ell^2(\Y))$ (this operator is denoted by $M_\psi A$). If $\psi$ is a Schur multiplier it follows easily from the closed graph theorem that $M_\psi\in\linbeg(\linbeg(\ell^2(\Y)))$, and one referrers to $\|M_\psi\|$ as the \emph{Schur norm} of $\psi$ and denotes it by $\|\psi\|_S$.
The following result, which gives a characterization of the Schur multipliers, is essentially due to Grothendieck, cf.~\cite[Theorem~5.1]{Pis:SimilarityProblemsAndCompletelyBoundedMaps} for a proof.
\begin{proposition}[Grothendieck]
	\label{Grothendieck}
	Let $\Y$ be a non-empty set and assume that $\psi:\Y\times\Y\to\C$ and $k\geq0$ are given, then the following are equivalent:
	\begin{itemize}
		\item [(i)]$\psi$ is a Schur multiplier with $\|\psi\|_S\leq k$.
		\item [(ii)]There exists a Hilbert space $\Hil$ and two bounded maps $P,Q:\Y\to\Hil$ such that
		\begin{equation*}
			\psi(x,y)=\ip{P(x)}{Q(y)}\qquad(x,y\in\Y)
		\end{equation*}
		and
		\begin{equation*}
			\|P\|_\infty\|Q\|_\infty\leq k,
		\end{equation*}
		where
		\begin{equation*}
			\|P\|_\infty=\sup_{x\in\Y}\|P(x)\|\quad\mbox{and}\quad\|Q\|_\infty=\sup_{y\in\Y}\|Q(y)\|.
		\end{equation*}
	\end{itemize}
\end{proposition}
It follows from (the proof of) the above theorem that $M_\psi$ is completely bounded when $\psi$ is a Schur multiplier and that $\|M_\psi\|_\cb=\|M_\psi\|$.

Let $\X$ be (the vertices of) a \emph{homogeneous tree} of degree $q+1$ for $2\leq q\leq\infty$, i.e., $\X$ consists of the vertices of a connected and cycle-free graph satisfying that each edge is connected to precisely $q+1$ other edges. Let $\dd:\X\times\X\to\N_0$ be the graph distance on $\X$, that is, $\dd(x,y)=1$ if and only if there is an edge connecting $x$ and $y$. Let $x_0$ be a fixed vertex in $\X$ and consider the pair $(\X,x_0)$. If $\varphi:\X\to\C$ is \emph{radial}, i.e., of the form
\begin{equation}
	\label{0.1}
	\varphi(x)=\dot\varphi(\dd(x,x_0))\qquad(x\in\X)
\end{equation}
for some $\dot\varphi:\N_0\to\C$, then we consider the function $\tilde\varphi:\X\times\X\to\C$ given by
\begin{equation}
	\label{0.2}
	\tilde\varphi(x,y)=\dot\varphi(\dd(x,y))\qquad(x,y\in\X).
\end{equation}

The main results of section~\ref{homtree} (Theorem~\ref{maintheorem} and~\ref{maintheoreminf}) are stated in Theorem~\ref{Theorem0.2} below:
\begin{theorem}
	\label{Theorem0.2}
	Let $(\X,x_0)$ be a homogeneous tree of degree $q+1$ ($2\leq q\leq\infty$) with a distinguished vertex $x_0\in\X$. Let $\varphi:\X\to\C$ be a radial function and let $\dot\varphi:\N_0\to\C$ and $\tilde\varphi:\X\times\X\to\C$ be defined as in~\eqref{0.1} and~\eqref{0.2}. Then $\tilde\varphi$ is a Schur multiplier if and only if the Hankel matrix $H=(h_{i,j})_{i,j\in\N_0}$ given by
	\begin{equation*}
		h_{i,j}=\dot\varphi(i+j)-\dot\varphi(i+j+2)\qquad(i,j\in\N_0)
	\end{equation*}
	is of trace class. In this case, the limits
	\begin{equation*}
		\lim_{n\to\infty}\dot\varphi(2n)\quad\mbox{and}\quad\lim_{n\to\infty}\dot\varphi(2n+1)
	\end{equation*}
	exists and the Schur norm of $\tilde\varphi$ is given by
	\begin{equation*}
		\|\tilde\varphi\|_{S}=|c_+|+|c_-|+\left\{
		\begin{array}{lll}
			\|H\|_1 & \mbox{if} & q=\infty\\
			\big(1-\tfrac{1}{q}\big)\|\big(I-\tfrac{\tau}{q}\big)^{-1}H\|_1 & \mbox{if} & q<\infty,
		\end{array}
		\right.
	\end{equation*}
	where
	\begin{equation*}
		c_\pm=\tfrac{1}{2}\lim_{n\to\infty}\dot\varphi(2n)\pm\tfrac{1}{2}\lim_{n\to\infty}\dot\varphi(2n+1)
	\end{equation*}
	and $\tau$ is the operator on the space of trace class operators $\traceclass(\ell^2(\N_0))$ given by
	\begin{equation*}
		\tau(A)=SAS^*\qquad(A\in\traceclass(\ell^2(\N_0))),
	\end{equation*}
	where $S$ is the forward shift on $\ell^2(\N_0)$.
\end{theorem}

In section~\ref{sphfct} we consider \emph{spherical functions} on a homogeneous tree $\X$ of degree $q+1$ ($2\leq q\leq\infty$). For $q<\infty$ the spherical functions can be characterized as the normalized radial eigenfunctions to the \emph{Laplace operator} $L$ (cf.~Definition~\ref{spherical}). Spherical functions have been studied extensively in the literature, cf.~\cite{FTN:HarmonicAnalysisAndRepresentationTheoryForGroupsActingOnHhomogeneousTrees}. Although the Laplace operator is not well defined for $q=\infty$ one can still define spherical functions in this case (cf.~Definition~\ref{sphericalinf}). The main result of section~\ref{sphfct} is the following characterization of the spherical functions $\varphi:\X\to\C$ for which the corresponding function $\tilde\varphi:\X\times\X\to\C$ is a Schur multiplier (cf.~Theorem~\ref{sphSchurNormEigenvalue} and~\ref{sphSchurNorminf}):
\begin{theorem}
	\label{Theorem0.3}
	Let $(\X,x_0)$ be a homogeneous tree of degree $q+1$ ($2\leq q\leq\infty$) with a distinguished vertex $x_0\in\X$. Let $\varphi:\X\to\C$ be a spherical function and let $\tilde\varphi:\X\times\X\to\C$ be the corresponding function as in~\eqref{0.2}. Then $\tilde\varphi$ is a Schur multiplier if and only if the eigenvalue $\eval$ corresponding to $\varphi$ is in the set
	\begin{equation*}
		\setw{\eval\in\C }{ \Re(\eval)^2+\big(\tfrac{q+1}{q-1}\big)^2\Im(\eval)^2<1 }\bigcup\seto{\pm1}.
	\end{equation*}
	The corresponding Schur norm is given by
	\begin{equation*}
		\|\tilde\varphi\|_S=\frac{|1-\eval^2|}{1-\Re(\eval)^2-\left(\frac{q+1}{q-1}\right)^2\Im(\eval)^2}\qquad(\Re(\eval)^2+\big(\tfrac{q+1}{q-1}\big)^2\Im(\eval)^2<1)
	\end{equation*}
	and
	\begin{equation*}
		\|\tilde\varphi\|_S=1\qquad(\eval=\pm1),
	\end{equation*}
	where we set $\frac{q+1}{q-1}$ equal to $1$ when $q=\infty$.
\end{theorem}
In section~\ref{integral} we use Theorem~\ref{Theorem0.2} together with a variant of Peller's characterization of Hankel operators of trace class (cf.~\cite[Theorem~1']{Pel:HankelOperatorsOfClassSpAndTheirApplications(rationalApproximationGaussianProcessesTheProblemOfMajorizationOfOperators)}) to obtain an integral representation of radial Schur multipliers on a homogeneous tree of degree $q+1$ ($2\leq q\leq\infty$), cf.~Theorem~\ref{integralthm} and Remark~\ref{2.3.3}.

Let $G$ be a locally compact group. In~\cite{Her:UneGeneralisationDeLaNotionDetransformeeDeFourier-Stieltjes}, Herz introduced a class of functions on $\G$, which was later denoted the class of \emph{Herz--Schur multipliers} on $\G$. By the introduction to~\cite{BF:Herz-SchurMultipliersAndCompletelyBoundedMultipliersOfTheFourierAlgebraOfALocallyCompactGroup}, a continuous function $\varphi:\G\to\C$ is a Herz--Schur multiplier if and only if the function
\begin{equation}
	\label{new0.3}
	\hat\varphi(x,y)=\varphi(y^{-1}x)\qquad(x,y\in\G)
\end{equation}
is a Schur multiplier, and the \emph{Herz--Schur norm} of $\varphi$ is given by
\begin{equation*}
	\|\varphi\|_{HS}=\|\hat\varphi\|_S.
\end{equation*}

In~\cite{DCH:MultipliersOfTheFourierAlgebrasOfSomeSimpleLieGroupsAndTheirDiscreteSubgroups} De Canni{\`e}re and Haagerup introduced the Banach algebra $\MA(G)$ of \emph{Fourier multipliers} of $G$, consisting of functions $\varphi:\G\to\C$ such that
\begin{equation*}
	\varphi\psi\in\FA(\G)\qquad(\psi\in\FA(\G)),
\end{equation*}
where $\FA(\G)$ is the \emph{Fourier algebra} of $\G$ as introduced by Eymard in~\cite{Eym:L'alg`ebreDeFourierD'unGroupeLocalementCompact} (the \emph{Fourier--Stieltjes algebra} $\FSA(\G)$ of $\G$ is also introduced in this paper). The norm of $\varphi$ (denoted \emph{$\|\varphi\|_{\MA(\G)}$}) is given by considering $\varphi$ as an operator on $\FA(\G)$. According to~\cite[Proposition~1.2]{DCH:MultipliersOfTheFourierAlgebrasOfSomeSimpleLieGroupsAndTheirDiscreteSubgroups} a Fourier multiplier of $G$ can also be characterized as a continuous function $\varphi:\G\to\C$ such that
\begin{equation*}
	\lambda(g)\stackrel{M_\varphi}{\mapsto}\varphi(g)\lambda(g)\qquad(g\in\G)
\end{equation*}
extends to a $\sigma$-weakly continuous operator (still denoted $M_\varphi$) on the group von Neumann algebra ($\lambda:\G\to\linbeg(\ELL^2(\G))$ is the \emph{left regular representation} and the group von Neumann algebra is the closure of the span of $\lambda(\G)$ in the weak operator topology). Moreover, one has $\|\varphi\|_{\MA(\G)}=\|M_\varphi\|$. The Banach algebra $\MoA(\G)$ of \emph{completely bounded Fourier multipliers} of $G$ consists of the Fourier multipliers of $G$, $\varphi$, for which $M_\varphi$ is completely bounded. In this case they put \emph{$\|\varphi\|_{\MoA(\G)}=\|M_\varphi\|_{\cb}$}.

In~\cite{BF:Herz-SchurMultipliersAndCompletelyBoundedMultipliersOfTheFourierAlgebraOfALocallyCompactGroup} Bo{\.z}ejko and Fendler show that the completely bounded Fourier multipliers coincide isometrically with the continuous Herz--Schur multipliers. In~\cite{Jol:ACharacterizationOfCompletelyBoundedMultipliersOfFourierAlgebras} Jolissaint gives a short and self-contained proof of the result from~\cite{BF:Herz-SchurMultipliersAndCompletelyBoundedMultipliersOfTheFourierAlgebraOfALocallyCompactGroup} in the form stated below.
\begin{proposition}[\cite{BF:Herz-SchurMultipliersAndCompletelyBoundedMultipliersOfTheFourierAlgebraOfALocallyCompactGroup},~\cite{Jol:ACharacterizationOfCompletelyBoundedMultipliersOfFourierAlgebras}]
	\label{Gilbert0}
	Let $G$ be a locally compact group and assume that $\varphi:G\to\C$ and $k\geq0$ are given, then the following are equivalent:
	\begin{itemize}
		\item [(i)]$\varphi$ is a completely bounded Fourier multiplier of $\G$ with $\|\varphi\|_{\MoA(G)}\leq k$.
		\item [(ii)]$\varphi$ is a continuous Herz--Schur multiplier on $\G$ with $\|\varphi\|_{HS}\leq k$.
		\item [(iii)]There exists a Hilbert space $\Hil$ and two bounded, continuous maps $P,Q:G\to\Hil$ such that
		\begin{equation*}
			\varphi(y^{-1}x)=\ip{P(x)}{Q(y)}\qquad(x,y\in G)
		\end{equation*}
		and
		\begin{equation*}
			\|P\|_\infty\|Q\|_\infty\leq k,
		\end{equation*}
		where
		\begin{equation*}
			\|P\|_\infty=\sup_{x\in G}\|P(x)\|\quad\mbox{and}\quad\|Q\|_\infty=\sup_{y\in G}\|Q(y)\|.
		\end{equation*}
	\end{itemize}
\end{proposition}

Consider the (non-abelian) free groups $\F_N$ ($2\leq N\leq\infty$), or more generally, groups of the form
\begin{equation}
	\label{a20.3}
	\Gamma=\FF,
\end{equation}
where $M,N\in\N_0\bigcup\{\infty\}$ and $q=M+2N-1\geq2$. The Cayley graph of $\Gamma$ is a homogeneous tree of degree $q+1$ (cf.~\cite[page~16--18]{FTN:HarmonicAnalysisAndRepresentationTheoryForGroupsActingOnHhomogeneousTrees}) with distinguished vertex $x_0=e$, the identity in $\Gamma$. Spherical functions on finitely generated free groups were introduced in~\cite{FTP:SphericalFunctionsAndHarmonicAnalysisOnFreeGroups}, \cite{FTP:HarmonicAnalysisOnFreeGroups}, and they were later generalized to groups $\Gamma$ of the form~\eqref{a20.3} with $q<\infty$ (cf.~\cite[Ch.~2]{FTN:HarmonicAnalysisAndRepresentationTheoryForGroupsActingOnHhomogeneousTrees}). The spherical functions on $\Gamma$ are simply the spherical functions on the homogeneous tree $(\Gamma,e)$, where we have identified (the vertices of) the Cayley graph with $\Gamma$. In section~\ref{fn} we use Theorem~\ref{Theorem0.2} and~\ref{Theorem0.3} to prove similar results about Fourier multipliers and spherical functions on groups $\Gamma$ of the form~\eqref{a20.3} (cf.~Theorem~\ref{maintheoremgamma} and~\ref{fnthm}). In particular, we obtain from Theorem~\ref{Theorem0.2}:
\begin{theorem}
	\label{Theorem0.5}
	Let $\Gamma$ be a group of the form~\eqref{a20.3} with $2\leq q\leq\infty$. Let $\varphi:\Gamma\to\C$ be a radial function and let $\dot\varphi:\N_0\to\C$ be the function defined by~\eqref{0.1}. Then $\varphi\in\MoA(\Gamma)$ if and only if the Hankel matrix $H=(h_{i,j})_{i,j\in\N_0}$ given by
	\begin{equation*}
		h_{i,j}=\dot\varphi(i+j)-\dot\varphi(i+j+2)\qquad(i,j\in\N_0)
	\end{equation*}
	is of trace class. In this case
	\begin{equation*}
		\|\varphi\|_{\MoA(\Gamma)}=|c_+|+|c_-|+\left\{
		\begin{array}{lll}
			\|H\|_1 & \mbox{if} & q=\infty\\
			\big(1-\tfrac{1}{q}\big)\|\big(I-\tfrac{\tau}{q}\big)^{-1}H\|_1 & \mbox{if} & q<\infty,
		\end{array}
		\right.,
	\end{equation*}
	where $c_\pm$ and $\tau$ are defined as in Theorem~\ref{Theorem0.2}.
\end{theorem}
Moreover, we use Theorem~\ref{Theorem0.5} to construct radial functions in $\MA(\Gamma)\setminus\MoA(\Gamma)$ for all groups $\Gamma$ of the form~\eqref{a20.3} (cf.~Proposition~\ref{existsradial}). Bo{\.z}ejko proved in~\cite{Boz:RemarkOnHerz-SchurMultipliersOnFreeGroups} that $\MA(\Gamma)\setminus\MoA(\Gamma)\neq\emptyset$ for the non-abelian free groups by constructing a non-radial function in this set.

For a prime number $q$ let $\padicnum$ denote the \emph{p-adic numbers} (corresponding to $q$) and let $\padicnums$ denote the invertible p-adic numbers (the non-zero p-adic numbers). Similarly, let $\padicint$ denote the \emph{p-adic integers} (corresponding to $q$) and let $\padicints$ denote the invertible p-adic integers (the \emph{p-adic units}). Let $\PGLQ$ denote the quotient of $\GLQ$ by its center $\padicnums I$, where $\GLQ$ denotes the $2\times2$ invertible matrices with entries from $\padicnum$. Similarly, let $\PSLZ$ denote the quotient of $\SLZ$ by its center $\padicints I$. One can, according to Serre (cf.~\cite[Chapter~II~\S1]{Ser:ArbresAmalgamesSL2}), interpret the quotient $\PGLQ/\PSLZ$ as a homogeneous tree $\X$ of degree $q+1$ with the range of the unit in $\PGLQ$ by the quotient map as distinguished vertex $x_0$. Moreover, $(\PGLQ,\PSLZ)$ form a Gelfand pair in the sense of~\cite{GV:HarmonicAnalysisOfSphericalFunctionsOnRealReductiveGroups} and there is a one-to-one correspondence between the spherical functions on $\PGLQ$ associated to this Gelfand pair and the spherical functions on the homogeneous tree $(\X,x_0)$ (cf.~Proposition~\ref{sphfctcoinside}). In section~\ref{pgl2qq} we use Theorem~\ref{Theorem0.2} and~\ref{Theorem0.3} to prove similar results for functions on $\PGLQ$ (cf.~Theorem~\ref{2.5.6} and Theorem~\ref{pgl2qqthm}). In particular, we obtain from Theorem~\ref{Theorem0.3}:
\begin{theorem}
	\label{Theorem1.3.1}
	Let $q$ be a prime number and consider the groups $\G=\PGLQ$ and $\K=\PSLZ$ and their quotient $\X=\G/\K$. Let $\varphi$ be a spherical function on the Gelfand pair $(\G,\K)$, then $\varphi$ is a completely bounded Fourier multiplier of $\G$ if and only if the eigenvalue $\eval$ of the corresponding spherical function on $\X$, is in the set
	\begin{equation*}
		\setw{\eval\in\C }{ \Re(\eval)^2+\big(\tfrac{q+1}{q-1}\big)^2\Im(\eval)^2<1 }\bigcup\seto{\pm1}.
	\end{equation*}
	The corresponding norm is given by
	\begin{equation*}
		\|\varphi\|_{\MoA(\G)}= \frac{|1-\eval^2|}{1-\Re(\eval)^2-\left(\frac{q+1}{q-1}\right)^2\Im(\eval)^2}\qquad(\Re(\eval)^2+\big(\tfrac{q+1}{q-1}\big)^2\Im(\eval)^2<1)
	\end{equation*}
	and
	\begin{equation*}
		\|\varphi\|_{\MoA(\G)}=1\qquad(\eval=\pm1).
	\end{equation*}
\end{theorem}

The present paper originates from an unpublished manuscript~\cite{HS:ManuscriptAboutRadialCompletelyBoundedMultipliersOnFreeGroups} from 1987 written by two of the authors of this paper. Thanks to the third author, the manuscript has now been largely extended in order to cover radial functions on homogeneous trees of arbitrary degree $q+1$ ($2\leq q\leq\infty$) as well as applications to the p-adic groups $\PGLQ$ for a prime number $q$. The original manuscript focussed on radial functions on the free groups $\F_N=\ast_{n=1}^N\Z$ ($2\leq N\leq\infty$). In particular, Theorem~\ref{Theorem0.5} was proved in~\cite{HS:ManuscriptAboutRadialCompletelyBoundedMultipliersOnFreeGroups} for the case $\Gamma=\F_N$. A few months after~\cite{HS:ManuscriptAboutRadialCompletelyBoundedMultipliersOnFreeGroups} was written, Bo{\.z}ejko included the proof of Theorem~\ref{Theorem0.5} in the case $\Gamma=\F_N$ in a set of (unpublished) lecture notes from Heidelberg University, cf.~\cite{Boz:PositiveAndNegativeDefiniteKernelsOnDiscreteGroups}. Later, Wysocza{\'n}ski obtained in~\cite{Wys:ACharacterizationOfRadialHerz-SchurMultipliersOnFreeProductsOfDiscreteGroups} a similar characterization of the radial Herz--Schur multipliers on a free product $\Gamma=\Gamma_1\ast\cdots\ast\Gamma_N$ ($2\leq N<\infty$) of $N$ groups of the same cardinality $k$ ($2\leq k\leq\infty$). The length function used in~\cite{Wys:ACharacterizationOfRadialHerz-SchurMultipliersOnFreeProductsOfDiscreteGroups} is the so-called block length of a reduced word in $\Gamma$.

	\section{Radial Schur multipliers on homogeneous trees}
\label{homtree}
Let $\X$ be (the vertices of) a homogeneous tree of degree $q+1$ for $2\leq q\leq\infty$, and consider the pair $(\X,x_0)$ where $x_0$ is a distinguished vertex in $\X$.
\begin{proposition}
	\label{radial}
	There is a bijective correspondence between the following types of functions:
	\begin{itemize}
		\item [(i)]$\dot\varphi:\N_0\to\C$.
		\item [(ii)]$\varphi:\X\to\C$ of the form
		\begin{equation*}
			\varphi(x)=\dot\varphi(\dd(x,x_0))\qquad(x\in\X)
		\end{equation*}
		for some $\dot\varphi:\N_0\to\C$.
		\item [(iii)]$\tilde\varphi:\X\times\X\to\C$ of the form
		\begin{equation*}
			\tilde\varphi(x,y)=\dot\varphi(\dd(x,y))\qquad(x,y\in\X)
		\end{equation*}
	for some $\dot\varphi:\N_0\to\C$.
	\end{itemize}
\end{proposition}
\begin{proof}
	This is obvious.
\end{proof}
A function of the type {\it (ii)} from Proposition~\ref{radial} is refereed to as a \emph{radial} function.

Let $S$ be the \emph{forward shift} on $\ell^2(\N_0)$, i.e.,
\begin{equation*}
	Se_n=e_{n+1}\qquad(n\in\N_0),
\end{equation*}
where $(e_n)_{n\in\N_0}$ is the canonical basis of $\ell^2(\N_0)$. Recall that $S^*S$ is the identity operator $I$ on $\ell^2(\N_0)$ and $SS^*$ is the projection on $\{e_0\}^\perp$.

Denote by $\|\cdot\|_1$ the norm on the \emph{trace class operators} $\traceclass(\ell^2(\N_0))$, i.e.,
\begin{equation*}
	\|T\|_1=\trace(|T|)=\sum_{n=0}^\infty\ip{|T|e_n}{e_n}
\end{equation*}
for any $T\in\linbeg(\ell^2(\N_0))$ for which this is finite.

Let $\tau\in\linbeg(\linbeg(\ell^2(\N_0)))$ be given by
\begin{equation}
	\label{tau}
	\tau(A)=SAS^*\qquad(A\in\linbeg(\ell^2(\N_0))).
\end{equation}
Obviously, $\tau$ is an isometry on the bounded operators. The following argument shows that $\tau$ is also an isometry on the trace class operators. If $T$ is a trace class operator on $\ell^2(\N_0)$ and $T=U|T|$ is the polar decomposition of $T$, then $\tau(T)=SUS^*S|T|S^*$ is the polar decomposition of $\tau(T)$, from which it follows that $\|\tau(T)\|_1=\trace(S|T|S^*)=\trace(|T|)=\|T\|_1$. This leads us to defining
\begin{equation*}
	\Big(I-\frac{\tau}{\alpha}\Big)^{-1}A=\sum_{n=0}^\infty \frac{\tau^n(A)}{\alpha^n}\qquad(\alpha>1,\,A\in\traceclass(\ell^2(\N_0))),
\end{equation*}
from which we see that $\big(I-\frac{\tau}{\alpha}\big)^{-1}$ makes sense as an element of $\linbeg(\traceclass(\ell^2(\N_0)))$, and its norm is bounded by $(1-\tfrac{1}{\alpha})^{-1}$.

Assume for now that $2\leq q<\infty$. For $m,n\in\N_0$ put
\begin{eqnarray*}
	S_{m,n} & = & \big(1-\tfrac{1}{q}\big)^{-1}(S^m(S^*)^n-\tfrac{1}{q}S^*S^m(S^*)^nS)\\
	& = & \left\{
	\begin{array}{lll}
		\big(1-\tfrac{1}{q}\big)^{-1}(S^m(S^*)^n-\tfrac{1}{q}S^{m-1}(S^*)^{n-1}) & \mbox{if} & m,n\geq1\\
		S^m(S^*)^n & \mbox{if} & \min\seto{m,n}=0
	\end{array}
	\right.
\end{eqnarray*}
Note that
\begin{equation*}
	S^m(S^*)^n=S_{m,n}\qquad(\min\seto{m,n}=0)
\end{equation*}
and
\begin{equation*}
	S^m(S^*)^n=\big(1-\tfrac{1}{q}\big)S_{m,n}+\tfrac{1}{q}S^{m-1}(S^*)^{n-1}\qquad(m,n\geq1).
\end{equation*}
Hence it follows by induction in $\min\seto{m,n}$ that $S^m(S^*)^n\in\sspan\setw{S_{k,l} }{ k,l\in\N_0}$ for all $m,n\in\N_0$. Since $\cstar(S)$ is the closed linear span of $\setw{S^m(S^*)^n }{ m,n\in\N_0}$ we also have
\begin{equation}
	\label{Sspan}
	\cstar(S)=\overline{\sspan}\setw{S_{m,n} }{ m,n\in\N_0}.
\end{equation}
\begin{lemma}
	\label{HH}
	Let $T,T'\in\linbeg(\ell^2(\N_0))$ be related by
	\begin{equation*}
		T'=\big(1-\tfrac{1}{q}\big)\big(I-\tfrac{\tau}{q}\big)^{-1}T.
	\end{equation*}
	Assume that one, and hence both, matrices are of trace class, then
	\begin{equation*}
		\trace(S^i(S^*)^jT)=\trace(S_{i,j}T')\qquad(i,j\in\N_0).
	\end{equation*}
\end{lemma}
\begin{proof}
	For $i,j\in\N_0$ we have that
	\begin{eqnarray*}
		\trace(S^i(S^*)^jT) & = & \big(1-\tfrac{1}{q}\big)^{-1}\trace(S^i(S^*)^jT'-\tfrac{1}{q}S^i(S^*)^j\tau(T'))\\
		& = & \big(1-\tfrac{1}{q}\big)^{-1}\trace(S^i(S^*)^jT'-\tfrac{1}{q}S^i(S^*)^jST'S^*)\\
		& = & \big(1-\tfrac{1}{q}\big)^{-1}\trace((S^i(S^*)^j-\tfrac{1}{q}S^*S^i(S^*)^jS)T')\\
		& = & \trace(S_{i,j}T'),
	\end{eqnarray*}
	which finishes the proof.
\end{proof}
\begin{maintheorem}
	\label{maintheorem}
	Let $(\X,x_0)$ be a homogeneous tree of degree $q+1$ ($2\leq q<\infty$) with a distinguished vertex $x_0\in\X$. Let $\varphi:\X\to\C$ be a radial function and let $\dot\varphi:\N_0\to\C$ and $\tilde\varphi:\X\times\X\to\C$ be the corresponding functions as in Proposition~\ref{radial}. Finally, let $H=(h_{i,j})_{i,j\in\N_0}$ be the Hankel matrix given by $h_{i,j}=\dot\varphi(i+j)-\dot\varphi(i+j+2)$ for $i,j\in\N_0$. Then the following are equivalent:
	\begin{itemize}
		\item [(i)]$\tilde\varphi$ is a Schur multiplier.
		\item [(ii)]$H$ is of trace class.
	\end{itemize}
	If these two equivalent conditions are satisfied, then there exists unique constants $c_\pm\in\C$ and a unique $\dot\psi:\N_0\to\C$ such that
	\begin{equation*}
		\dot\varphi(n)=c_++c_-(-1)^n+\dot\psi(n)\qquad(n\in\N_0)
	\end{equation*}
	and
	\begin{equation*}
		\lim_{n\to\infty}\dot\psi(n)=0.
	\end{equation*}
	Moreover,
	\begin{equation*}
		\|\tilde\varphi\|_S=|c_+|+|c_-|+\big(1-\tfrac{1}{q}\big)\|\big(I-\tfrac{\tau}{q}\big)^{-1}H\|_1,
	\end{equation*}
	where $\tau$ is the shift operator defined by~\eqref{tau}.
\end{maintheorem}
In order to prove Theorem~\ref{maintheorem}, choose (once and for all) an infinite chain $\omega$ in $\X$ starting at $x_0$, i.e., an infinite sequence $x_0,x_1,x_2,\ldots$ such that $x_i$ and $x_{i+1}$ are connected by an edge and $x_i\neq x_{i+2}$ for all $i\in\N_0$ (cf.~\cite[Chapter~I~\S1]{FTN:HarmonicAnalysisAndRepresentationTheoryForGroupsActingOnHhomogeneousTrees}). Since $\X$ is a tree we have $x_i\neq x_j$ whenever $i\neq j$. Define a map $c:\X\to\X$ such that for any $x\in\X$ the sequence $x,c(x),c^2(x),\ldots$ becomes the infinite chain setting out at $x$ and eventually following $\omega$ (this chain is denoted by $[x,\omega)$ in \cite{FTN:HarmonicAnalysisAndRepresentationTheoryForGroupsActingOnHhomogeneousTrees}). To make this more precise, define
\begin{equation*}
	c(x)=\left\{
	\begin{array}{lll}
		x_{i+1} & \mbox{if} & x=x_i\ \mbox{for some}\ i\in\N_0\\
		x' & \mbox{if} & x\neq x_i\ \mbox{for every}\ i\in\N_0
	\end{array}
	\right.
	\qquad(x\in\X),
\end{equation*}
where $x'$ is the unique vertex satisfying $\dd(x,x')=1$ and $\dd(x',\omega)=\dd(x,\omega)-1$, and where $\dd(y,\omega)=\min\setw{\dd(y,x_i) }{ i\in\N_0}$ for $y\in\X$.
\begin{remark}
	\label{xymn}
	For $x,y\in\X$ there are smallest numbers $m,n\in\N_0$ such that $\cc^m(x)\in[y,\omega)$ and $\cc^n(y)\in[x,\omega)$. Moreover, these $m,n\in\N_0$ can be characterized as the unique numbers satisfying
	\begin{equation*}
		\cc^m(x)=\cc^n(y)\quad\mbox{and}\quad\cc^{m-1}(x)\neq\cc^{n-1}(y)\quad\mbox{if}\quad m,n\geq1,
	\end{equation*}
	and
	\begin{equation*}
		\cc^m(x)=\cc^n(y)\quad\mbox{if}\quad\min\seto{m,n}=0.
	\end{equation*}
	Note that in both cases $\dd(x,y)=m+n$.
\end{remark}
Put
\begin{equation*}
	U\delta_x=\frac{1}{\sqrt q}\sum_{\cc(z)=x}\delta_z\qquad(x\in\X)
\end{equation*}
and observe that $\setw{z\in\X }{ \cc(z)=x }$ consists of precisely $q$ elements, because this set contains all neighbor points to $x$ except $\cc(x)$. Since two such sets $\setw{z\in\X }{ \cc(z)=x }$, $\setw{z\in\X }{ \cc(z)=x' }$ are disjoint if $x\neq x'$, it follows that $(U\delta_x)_{x\in\X}$ is an orthonormal set in $\ell^2(\X)$. This shows that $U$ extends to an isometry of $\ell^2(\X)$. Elementary computations show that
\begin{equation*}
	U^*\delta_x=\frac{1}{\sqrt q}\delta_{\cc(x)}\qquad(x\in\X)
\end{equation*}
and
\begin{equation*}
	UU^*\delta_x=\frac{1}{q}\sum_{\cc(z)=\cc(x)}\delta_z\qquad(x\in\X).
\end{equation*}
In particular, $UU^*\neq I$ so $U$ is a non-unitary isometry. For each $x\in\X$ we define a vector $\delta_x'\in\ell^2(\X)$ by
\begin{equation*}
	\delta_x'=\big(1-\tfrac{1}{q}\big)^{-\tfrac{1}{2}}(I-UU^*)\delta_x=\big(1-\tfrac{1}{q}\big)^{-\tfrac{1}{2}}\big(\delta_x-\frac{1}{q}\sum_{\cc(z)=\cc(x)}\delta_z\big)\qquad(x\in\X).
\end{equation*}
Using the fact that for all $w\in\X$ the set $\setw{z\in\X }{ \cc(z)=w}$ has $q$ elements, one easily checks that
\begin{equation}
	\label{ipdelta}
	\ip{\delta_y'}{\delta_x'}=\left\{
	\begin{array}{lll}
		1 & \mbox{if} & x=y\\
		-\tfrac{1}{q-1} & \mbox{if} & x\neq y,\,\cc(x)=\cc(y)\\
		0 & \mbox{if} & \cc(x)\neq\cc(y)
	\end{array}
	\right.
	\qquad(x,y\in\X).
\end{equation}
\begin{lemma}
	\label{Smn}
	For $x,y\in\X$ we have that
	\begin{equation*}
		(S_{m,n})_{i,j}=\ip{\delta_{\cc^j(y)}'}{\delta_{\cc^i(x)}'}\qquad(i,j\in\N_0),
	\end{equation*}
	when $m,n\in\N_0$ are chosen as in Remark~\ref{xymn}.
\end{lemma}
\begin{proof}
	For $m,n,i,j\in\N_0$ we have that
	\begin{eqnarray*}
		(S_{m,n})_{i,j} & = & \ip{S_{m,n}e_j}{e_i}\\
		& = & \big(1-\tfrac{1}{q}\big)^{-1}\ip{[S^m(S^*)^n-\tfrac{1}{q}S^*S^m(S^*)^nS]e_j}{e_i}\\
		& = & \big(1-\tfrac{1}{q}\big)^{-1}[\ip{(S^*)^n e_j}{(S^*)^m e_i}-\tfrac{1}{q}\ip{(S^*)^n e_{j+1}}{(S^*)^m e_{i+1}}]\\
		& = & \left\{
		\begin{array}{lll}
			1 & \mbox{if} & i-m=j-n\geq0\\
			-\tfrac{1}{q-1} & \mbox{if} & i-m=j-n=-1\\
			0 & \mbox{if} & i-m=j-n<-1\quad\mbox{or}\quad i-m\neq j-n
		\end{array}
		\right..
	\end{eqnarray*}
	On the other hand, if $x,y\in\X$ and $m,n\in\N_0$ are defined according to Remark~\ref{xymn}, then by~\eqref{ipdelta},
	\begin{equation*}
		\ip{\delta_{\cc^j(y)}'}{\delta_{\cc^i(x)}'}=\left\{
		\begin{array}{lll}
			1 & \mbox{if} & \cc^j(y)=\cc^i(x)\\
			-\tfrac{1}{q-1} & \mbox{if} & \cc^j(y)\neq\cc^i(x)\quad\mbox{and}\quad\cc^{j+1}(y)=\cc^{i+1}(x)\\
			0 & \mbox{if} & \cc^{j+1}(y)\neq\cc^{i+1}(x)
		\end{array}
		\right.
	\end{equation*}
	By the definition of $m$ and $n$ we have
	\begin{equation*}
		\cc^j(y)=\cc^i(x)\iff i-m=j-n\geq0
	\end{equation*}
	and
	\begin{equation*}
		\cc^{j+1}(y)=\cc^{i+1}(x)\iff i+1-m=j+1-n\geq0.
	\end{equation*}
	Therefore
	\begin{equation*}
		\ip{\delta_{\cc^j(y)}'}{\delta_{\cc^i(x)}'}=\left\{
		\begin{array}{lll}
			1 & \mbox{if} & i-m=j-n\geq0\\
			-\tfrac{1}{q-1} & \mbox{if} & i-m=j-n=-1\\
			0 & \mbox{if} & i-m=j-n<-1\quad\mbox{or}\quad i-m\neq j-n
		\end{array}
		\right.
	\end{equation*}
	This proves Lemma~\ref{Smn}.
\end{proof}
Similarly to how we defined $S_{m,n}$ for $m,n\in\N_0$, put
\begin{eqnarray}
	\label{UU}\\
	\nonumber U_{m,n} & = & \big(1-\tfrac{1}{q}\big)^{-1}(U^m(U^*)^n-\tfrac{1}{q}U^*U^m(U^*)^nU)\\
	\nonumber & = & \left\{
	\begin{array}{lll}
		\big(1-\tfrac{1}{q}\big)^{-1}(U^m(U^*)^n-\tfrac{1}{q}U^{m-1}(U^*)^{n-1}) & \mbox{if} & m,n\geq1\\
		U^m(U^*)^n & \mbox{if} & \min\seto{m,n}=0
	\end{array}
	\right..
\end{eqnarray}
According to Coburn's theorem (cf.~\cite[Theorem~3.5.18]{Mur:C*-algebrasAndOperatorTheory}) there exists a $*$-isomorphism $\Phi$ of $\cstar(S)$ onto $\cstar(U)$ such that $\Phi(S)=U$. Hence, by~\eqref{Sspan}, $\cstar(U)$ is equal to the closed linear span of $\setw{U_{m,n} }{ m,n\in\N_0}$.
\begin{lemma}
	\label{Umn}
	For $x,y\in\X$ we have that $(U_{m,n})_{x,y}$ is non-zero if and only if $m,n\in\N_0$ are chosen as in Remark~\ref{xymn}. In particular, $(U_{m,n})_{x,y}\neq0$ implies that $\dd(x,y)=m+n$.
\end{lemma}
\begin{proof}
	Let $m,n\in\N_0$ and $x,y\in\X$. By~\eqref{UU} we have for $m,n\geq1$
	\begin{eqnarray*}
		(U_{m,n})_{x,y} & = & \big(1-\tfrac{1}{q}\big)^{-1}\big(\ip{U^m(U^*)^n\delta_y}{\delta_x}-\tfrac{1}{q}\ip{U^{m-1}(U^*)^{n-1}\delta_y}{\delta_x}\big)\\
		& = & \big(1-\tfrac{1}{q}\big)^{-1}q^{-\tfrac{m+n}{2}}\big(\ip{\delta_{\cc^n(y)}}{\delta_{\cc^m(x)}}-\ip{\delta_{\cc^{n-1}(y)}}{\delta_{\cc^{m-1}(x)}}\big).
	\end{eqnarray*}
	Since $\cc^{n-1}(y)=\cc^{m-1}(x)\implies\cc^n(y)=\cc^m(x)$ and hence $\cc^n(y)\neq\cc^m(x)\implies\cc^{n-1}(y)\neq\cc^{m-1}(x)$ we find that
	\begin{equation*}
		(U_{m,n})_{x,y}=\big(1-\tfrac{1}{q}\big)^{-1}q^{-\tfrac{m+n}{2}}\quad\mbox{if}\quad\cc^n(y)=\cc^m(x)\quad\mbox{and}\quad\cc^{n-1}(y)\neq\cc^{m-1}(x),
	\end{equation*}
	and
	\begin{equation*}
		(U_{m,n})_{x,y}=0\quad\mbox{if}\quad\cc^n(y)\neq\cc^m(x)\quad\mbox{or}\quad\cc^{n-1}(y)=\cc^{m-1}(x).
	\end{equation*}
	If $\min\seto{m,n}=0$, then by~\eqref{UU}
	\begin{eqnarray*}
		(U_{m,n})_{x,y} & = & \ip{U^m(U^*)^n\delta_y}{\delta_x}\\
		& = & q^{-\tfrac{m+n}{2}}\ip{\delta_{\cc^n(y)}}{\delta_{\cc^m(x)}}\\
		& = & \left\{
		\begin{array}{lll}
			q^{-\tfrac{m+n}{2}} & \mbox{if} & \cc^n(y)=\cc^m(x)\\
			0 & \mbox{if} & \cc^n(y)\neq\cc^m(x)
		\end{array}
		\right.
	\end{eqnarray*}
	In both cases we see that $(U_{m,n})_{x,y}\neq0$ if and only if $m,n\in\N_0$ are defined from $x,y\in\X$ as in Remark~\ref{xymn}.
\end{proof}
\begin{corollary}
	\label{cstarU}
	Let $\varphi:\X\to\C$ be radial and $\tilde\varphi:\X\times\X\to\C$ the corresponding function as in Proposition~\ref{radial}. If $\tilde\varphi$ is a Schur multiplier, then $\cstar(U)$ is invariant under $M_{\tilde\varphi}\in\linbeg(\linbeg(\ell^2(\X)))$. Moreover,
	\begin{equation*}
		M_{\tilde\varphi}(U_{m,n})=\dot\varphi(m+n)U_{m,n}\qquad(m,n\in\N_0).
	\end{equation*}
\end{corollary}
\begin{proof}
	Since $\cstar(U)$ is equal to the closed linear span of $(U_{m,n})_{m,n\in\N_0}$ we only have to show that
	\begin{equation*}
		M_{\tilde\varphi}(U_{m,n})=\dot\varphi(m+n)U_{m,n}\qquad(m,n\in\N_0).
	\end{equation*}
	But from the definition of a Schur multiplier it follows that
	\begin{equation*}
		(M_{\tilde\varphi}(U_{m,n}))_{x,y}=\tilde\varphi(x,y)(U_{m,n})_{x,y}=\dot\varphi(m+n)(U_{m,n})_{x,y}\qquad(m,n\in\N_0),
	\end{equation*}
	since, according to Lemma~\ref{Umn}, $(U_{m,n})_{x,y}\neq0$ implies that $m+n=\dd(x,y)$.
\end{proof}
Following the notation of~\cite[3.3.9]{Ped:AnalysisNow} we let $\xi\odot\eta$ denote the rank one operator given by
\begin{equation*}
	(\xi\odot\eta)(\zeta)=\ip{\zeta}{\eta}\xi\qquad(\zeta\in\ell^2(\N_0))
\end{equation*}
for $\xi,\eta\in\ell^2(\N_0)$. It is elementary to check that the trace class norm of $\xi\odot\eta$ is
\begin{equation}
	\label{1x}
	\|\xi\odot\eta\|_1=\|\xi\|_2\|\eta\|_2\qquad(\xi,\eta\in\ell^2(\N_0)).
\end{equation}
If $\xi^{(k)},\eta^{(k)}\in\ell^2(\N_0)$ for all $k\in\N_0$ and
\begin{equation*}
	\sum_{k=0}^\infty\|\xi^{(k)}\|_2^2<\infty\quad\mbox{and}\quad\sum_{k=0}^\infty\|\eta^{(k)}\|_2^2<\infty,
\end{equation*}
then
\begin{equation*}
	T=\sum_{k=0}^\infty\xi^{(k)}\odot\eta^{(k)}
\end{equation*}
is a well defined trace class operator, because
\begin{equation}
	\label{finnorm}
	\sum_{k=0}^\infty\|\xi^{(k)}\|_2\|\eta^{(k)}\|_2\leq\Big(\sum_{k=0}^\infty\|\xi^{(k)}\|_2^2\Big)^{\tfrac{1}{2}}\Big(\sum_{k=0}^\infty\|\eta^{(k)}\|_2^2\Big)^{\tfrac{1}{2}}<\infty.
\end{equation}
Conversely, if $T\in\traceclass(\ell^2(\N_0))$ there exists sequences $(\xi^{(k)})_{k\in\N_0},(\eta^{(k)})_{k\in\N_0}$ in $\ell^2(\N_0)$ such that
\begin{equation}
	\label{1*}
	\sum_{k=0}^\infty\|\xi^{(k)}\|_2^2=\sum_{k=0}^\infty\|\eta^{(k)}\|_2^2=\|T\|_1<\infty
\end{equation}
and
\begin{equation}
	\label{1**}
	T=\sum_{k=0}^\infty\xi^{(k)}\odot\eta^{(k)}.
\end{equation}
Finally, note that~\eqref{1*} and~\eqref{1**} imply that
\begin{equation}
	\label{1***}
	\|T\|_1=\sum_{k=0}^\infty\|\xi^{(k)}\|_2\|\eta^{(k)}\|_2.
\end{equation}

This is well known, and it can be obtained from the polar decomposition $T=U|T|$ of $T$ combined with the spectral theorem for compact normal operators (cf.~\cite[Theorem~3.3.8]{Ped:AnalysisNow}), which shows that
\begin{equation*}
	|T|=\sum_{i\in I}\lambda_i e_i\odot e_i,
\end{equation*}
where $(e_i)_{i\in I}$ is an orthonormal basis of eigenvectors for $|T|$ and $(\lambda_i)_{i\in I}$ are the corresponding (non-negative) eigenvalues of $|T|$. Note that
\begin{equation*}
	\sum_{i\in I}\lambda_i=\trace(|T|)=\|T\|_1<\infty.
\end{equation*}
In particular, $I_0=\setw{i\in I }{ \lambda_i>0}$ is countable (possibly finite). Moreover,
\begin{equation*}
	T=\sum_{i\in I_0}\xi^{(i)}\odot\eta^{(i)},
\end{equation*}
where $\xi^{(i)}=(\lambda_i)^{\tfrac{1}{2}}Ue_i$ and $\eta^{(i)}=(\lambda_i)^{\tfrac{1}{2}}e_i$ satisfy
\begin{equation*}
	\sum_{i\in I_0}\|\xi^{(i)}\|_2^2=\sum_{i\in I_0}\|\eta^{(i)}\|_2^2=\sum_{i\in I_0}\lambda_i=\|T\|_1.
\end{equation*}
This proves~\eqref{1*} and~\eqref{1**} because $I_0$ is countable.
\begin{proof}[Proof of Theorem~\ref{maintheorem} {\it (ii)}$\implies${\it (i)} and upper bound for $\|\tilde\varphi\|_S$]
	Assuming that {\it (ii)} is true we have that the Hankel matrix $H=(h_{i,j})_{i,j\in\N_0}$ is of trace class. If $A$ is a trace class operator, then $A$ is a linear combination of positive trace class operators and therefore
	\begin{equation*}
		\sum_{n=0}^\infty |\ip{Ae_n}{e_n}|<\infty
	\end{equation*}
	and it follows that
	\begin{equation*}
		\sum_{i=0}^\infty|h_{i,i}|<\infty\quad\mbox{and}\quad\sum_{i=0}^\infty|h_{i+1,i}|<\infty
	\end{equation*}
	by putting $A=H$ and $A=S^*H$, respectively (note that $S^*H$ is of trace class since $H$ is of trace class). Using that
	\begin{equation*}
		h_{i,i}=\dot\varphi(2i)-\dot\varphi(2i+2)\quad\mbox{and}\quad h_{i+1,i}=\dot\varphi(2i+1)-\dot\varphi(2i+3)\qquad(i\in\N_0)
	\end{equation*}
	we conclude that
	\begin{equation*}
		\lim_{i\to\infty}\dot\varphi(2i)=\dot\varphi(0)-\sum_{i=0}^\infty h_{i,i}\quad\mbox{and}\quad\lim_{i\to\infty}\dot\varphi(2i+1)=\dot\varphi(1)-\sum_{i=0}^\infty h_{i+1,i},
	\end{equation*}
	where the sums converge (absolutely). Put
	\begin{equation*}
		c_\pm=\tfrac{1}{2}\lim_{i\to\infty}\dot\varphi(2i)\pm\tfrac{1}{2}\lim_{i\to\infty}\dot\varphi(2i+1)
	\end{equation*}
	and
	\begin{equation*}
		\dot\psi(n)=\dot\varphi(n)-c_+-c_-(-1)^n\qquad(n\in\N_0).
	\end{equation*}
	Notice that
	\begin{equation*}
		\lim_{n\to\infty}\dot\psi(n)=0.
	\end{equation*}
	We conclude the existence of $c_\pm$ and $\dot\psi$ as claimed in the theorem, and note that the uniqueness follows readily.
	
	Put
	\begin{equation*}
		H'=\big(1-\tfrac{1}{q}\big)\big(I-\tfrac{\tau}{q}\big)^{-1}H
	\end{equation*}
	and recall that
	\begin{equation*}
		\trace(S^i(S^*)^jH)=\trace(S_{i,j}H')\qquad(i,j\in\N_0)
	\end{equation*}
	according to Lemma~\ref{HH}. Since
	\begin{equation*}
		\trace(S^i(S^*)^jH)=\sum_{k=0}^\infty h_{k+j,k+i}\qquad(i,j\in\N_0)
	\end{equation*}
	it follows using
	\begin{equation*}
		h_{j,i}=\dot\varphi(i+j)-\dot\varphi(i+j+2)=\dot\psi(i+j)-\dot\psi(i+j+2)\qquad(i,j\in\N_0)
	\end{equation*}
	and
	\begin{equation*}
		\lim_{n\to\infty}\dot\psi(n)=0
	\end{equation*}
	that
	\begin{equation}
		\label{psitrace}
		\dot\psi(i+j)=\trace(S_{i,j}H')\qquad(i,j\in\N_0).
	\end{equation}
	Since $H'$ is of trace class, there exists (cf.~\eqref{1**} and~\eqref{1***}) sequences $(\xi^{(k)})_{k\in\N_0}$ and $(\eta^{(k)})_{k\in\N_0}$ in $\ell^2(\N_0)$ such that
	\begin{equation*}
		H'=\sum_{k=0}^\infty\xi^{(k)}\odot\eta^{(k)}\quad\mbox{and}\quad\|H'\|_1=\sum_{k=0}^\infty\|\xi^{(k)}\|_2\|\eta^{(k)}\|_2,
	\end{equation*}
	and therefore
	\begin{equation}
		\label{8.5}
		h_{i,j}'=\sum_{k=0}^\infty\xi_i^{(k)}\bar\eta_j^{(k)}\qquad(i,j\in\N_0).
	\end{equation}
	For each $k\in\N_0$ we define $P_k,Q_k:X\to\ell^2(\X)$ by
	\begin{equation*}
		P_k(x)=\sum_{i=0}^\infty\xi_i^{(k)}\delta_{\cc^i(x)}'\quad\mbox{and}\quad Q_k(y)=\sum_{j=0}^\infty\eta_j^{(k)}\delta_{\cc^j(y)}'\qquad(x,y\in\X).
	\end{equation*}
	By~\eqref{ipdelta}, $\setw{\delta_{\cc^i(x)}'}{i\in\N_0}$ and $\setw{\delta_{\cc^j(y)}'}{j\in\N_0}$ are orthonormal sets in $\ell^2(\X)$. Hence,
	\begin{equation*}
		\|P_k(x)\|_2=\|\xi^{(k)}\|_2\quad\mbox{and}\quad\|Q_k(y)\|_2=\|\eta^{(k)}\|_2\qquad(k\in\N_0,\,x,y\in\X),
	\end{equation*}
	and therefore
	\begin{equation*}
		\sum_{k=0}^\infty\|P_k\|_\infty\|Q_k\|_\infty=\sum_{k=0}^\infty\|\xi^{(k)}\|_2\|\eta^{(k)}\|_2=\|H'\|_1.
	\end{equation*}
	By~\eqref{8.5}
	\begin{equation*}
		\sum_{k=0}^\infty\ip{P_k(x)}{Q_k(y)}=\sum_{k,i,j=0}^\infty\ip{\delta_{\cc^i(x)}'}{\delta_{\cc^j(y)}'}\xi_i^{(k)}\bar\eta_j^{(k)}=\sum_{i,j=0}^\infty\ip{\delta_{\cc^i(x)}'}{\delta_{\cc^j(y)}'} h_{i,j}'
	\end{equation*}
	for all $x,y\in\X$. Momentarily fix $x,y\in\X$ and choose $m,n\in\N_0$ according to Remark~\ref{xymn}. Then $m+n=\dd(x,y)$ and by Lemma~\ref{Smn}
	\begin{equation*}
		(S_{m,n})_{j,i}=\ip{\delta_{\cc^i(x)}'}{\delta_{\cc^j(y)}'}\qquad(i,j\in\N_0).
	\end{equation*}
	Using~\eqref{psitrace} it follows that
	\begin{equation*}
		\sum_{k=0}^\infty\ip{P_k(x)}{Q_k(y)}=\sum_{i,j=0}^\infty (S_{m,n})_{j,i}h_{i,j}'=\trace(S_{m,n}H')=\dot\psi(m+n)=\tilde\psi(x,y).
	\end{equation*}
	Since $x,y\in\X$ were arbitrary we have that
	\begin{equation*}
		\tilde\varphi(x,y)=\dot\varphi(\dd(x,y))=c_++c_-(-1)^{\dd(x,y)}+\sum_{k=0}^\infty\ip{P_k(x)}{Q_k(y)}\qquad(x,y\in\X).
	\end{equation*}
	Put
	\begin{equation*}
		P_\pm(x)=c_\pm(\pm1)^{\dd(x,x_0)}\qquad(x\in\X)
	\end{equation*}
	and
	\begin{equation*}
		Q_\pm(y)=(\pm1)^{\dd(y,x_0)}\qquad(y\in\X),
	\end{equation*}
	then
	\begin{equation*}
		\tilde\varphi(x,y)=\ip{P_+(x)}{Q_+(y)}+\ip{P_-(x)}{Q_-(y)}+\sum_{k=0}^\infty\ip{P_k(x)}{Q_k(y)}\qquad(x,y\in\X)
	\end{equation*}
	and we conclude that $\tilde\varphi$ is a Schur multiplier with
	\begin{eqnarray*}
		\|\tilde\varphi\|_S & \leq & \|P_+\|_\infty\|Q_+\|_\infty+\|P_-\|_\infty\|Q_-\|_\infty+\sum_{k=0}^\infty\|P_k\|_\infty\|Q_k\|_\infty\\
		& = & |c_+|+|c_-|+\|H'\|_1.
	\end{eqnarray*}
	This finishes the first part of the proof of Theorem~\ref{maintheorem}.
\end{proof}
\begin{proposition}
	\label{functional}
	Let $V$ be a non-unitary isometry on some Hilbert space $\Hil$ and let $f$ be a bounded linear functional on $\cstar(V)$. Then there exists a complex Borel measure $\mu$ on $\T=\setw{z\in\C }{ |z|=1 }$ and a trace class operator $T$ on $\ell^2(\N_0)$ such that
	\begin{equation}
		\label{muT}
		f(V^m(V^*)^n)=\int_\T z^{m-n}\dd\mu(z)+\trace(S^m(S^*)^nT)\qquad(m,n\in\N_0).
	\end{equation}
	Moreover,
	\begin{equation*}
		\|f\|=\|\mu\|+\|T\|_1.
	\end{equation*}
\end{proposition}
\begin{proof}
	Let $(\pi,\Hil)$ be the universal representation of $\cstar(V)$. Then there exists $\xi,\eta\in\Hil$ such that
	\begin{equation*}
		f(A)=\ip{\pi(A)\xi}{\eta}\qquad(A\in\cstar(V))
	\end{equation*}
	and $\|f\|=\|\xi\|\|\eta\|$. By the Wold--von Neumann theorem (cf.~\cite[Theorem~3.5.17]{Mur:C*-algebrasAndOperatorTheory}), $\Hil$ can be decomposed as an orthogonal direct sum
	\begin{equation}
		\label{*}
		\Hil=K\oplus(\oplus_{e\in E} L_e),
	\end{equation}
	where $K$ and $(L_e)_{e\in E}$ are $V$-invariant closed subspaces, $V_0=V|_{K}$ is a unitary operator on $K$ and for each $e\in E$, $V_e=V|_{L_e}$ is a copy of the forward shift $S$ on $\ell^2(\N_0)$. We can decompose $\xi$ and $\eta$ according to~\eqref{*}:
	\begin{equation*}
		\xi=\xi_0\oplus(\oplus_{e\in E}\xi_e)\quad\mbox{and}\quad\eta=\eta_0\oplus(\oplus_{e\in E}\eta_e),
	\end{equation*}
	where
	\begin{equation*}
		\|\xi\|^2=\|\xi_0\|^2+\sum_{e\in E}\|\xi_e\|^2\quad\mbox{and}\quad\|\eta\|^2=\|\eta_0\|^2+\sum_{e\in E}\|\eta_e\|^2.
	\end{equation*}
	After identifying $(V_e,L_e)$ with $(S,\ell^2(\N_0))$, we have
	\begin{eqnarray*}
		f(V^m(V^*)^n) & = & \ip{V_0^{m-n}\xi_0}{\eta_0}+\sum_{e\in E}\ip{S^m(S^*)^n\xi_e}{\eta_e}\\
		& = & \ip{V_0^{m-n}\xi_0}{\eta_0}+\trace(S^m(S^*)^nT)
	\end{eqnarray*}
	for $m,n\in\N_0$, where $T=\sum_{e\in E}\xi_e\odot\eta_e\in\traceclass(\ell^2(\N_0))$.
	
	Since $V_0$ is a unitary operator we have a natural isomorphism $\cstar(V_0)\cong\cont(\sigma(V_0))$, where $\sigma(V_0)\subseteq\T$. Hence by the Riesz representation theorem, there exists a complex measure $\mu$ on $\T$ with $\support(\mu)\subseteq\sigma(V_0)$ such that
	\begin{equation*}
		\ip{V_0^k\xi_0}{\eta_0}=\int_\T z^k\dd\mu(z)\qquad(k\in\Z)
	\end{equation*}
	and $\|\mu\|\leq\|\xi_0\|\|\eta_0\|$. Hence
	\begin{equation}
		\label{**}
		f(V^m(V^*)^n)=\int_\T z^{m-n}\dd\mu(z)+\trace(S^m(S^*)^nT)\qquad(m,n\in\N_0)
	\end{equation}
	and
	\begin{eqnarray*}
		\|f\| & = & \big(\|\xi_0\|^2+\sum_{e\in E}\|\xi_e\|^2\big)^{\tfrac{1}{2}}\big(\|\eta_0\|^2+\sum_{e\in E}\|\eta_e\|^2\big)^{\tfrac{1}{2}}\\
		& \geq & \|\xi_0\|\|\eta_0\|+\sum_{e\in E}\|\xi_e\|\|\eta_e\|\\
		& \geq & \|\mu\|+\|T\|_1.
	\end{eqnarray*}
	The converse inequality $\|f\|\leq\|\mu\|+\|T\|_1$ follows from~\eqref{**}.
\end{proof}
\begin{lemma}
	\label{staralg}
	Let $\varphi:\X\to\C$ be radial and $\tilde\varphi:\X\times\X\to\C$ be the corresponding function as in Proposition~\ref{radial}. If $\tilde\varphi$ is a Schur multiplier, then there exists a bounded linear functional $f_\varphi$ on $\cstar(U)$ satisfying
	\begin{equation}
		\label{diamond}
		f_\varphi(U_{m,n})=\dot\varphi(m+n)\qquad(m,n\in\N_0)
	\end{equation}
	and
	\begin{equation}
		\label{wr}
		\|f_\varphi\|\leq\|\tilde\varphi\|_S.
	\end{equation}
\end{lemma}
\begin{proof}
	By Coburn's theorem (cf.~\cite[Theorem~3.5.18]{Mur:C*-algebrasAndOperatorTheory}) and~\cite[Remark~3.5.1]{Mur:C*-algebrasAndOperatorTheory} there exists a $*$-homomorphism $\rho$ of $\cstar(U)$ onto $\cont(\T)$ such that $\rho(U)(z)=z$ for $z\in\T$. Let $\gamma_0:\cont(\T)\to\C$ be the pure state given by
	\begin{equation*}
		\gamma_0(f)=f(1)\qquad(f\in\cont(\T)).
	\end{equation*}
	Then $\gamma=\gamma_0\circ\rho$ is a state on $\cstar(U)$ and
	\begin{equation*}
		\gamma(U^m(U^*)^n)=1\qquad(m,n\in\N_0).
	\end{equation*}
	Define $f_\varphi:\cstar(U)\to\C$ by
	\begin{equation*}
		f_\varphi(W)=\gamma(M_{\tilde\varphi}(W))\qquad(W\in\cstar(U)).
	\end{equation*}
	Then $f_\varphi\in\cstar(U)^*$, $\|f_\varphi\|\leq\|M_{\tilde\varphi}\|=\|\tilde\varphi\|_S$ and by Corollary~\ref{cstarU} and~\eqref{UU} we have
	\begin{equation*}
		f_\varphi(U_{m,n})=\dot\varphi(m+n)\gamma(U_{m,n})=\dot\varphi(m+n).
	\end{equation*}
\end{proof}
\begin{proof}[Proof of Theorem~\ref{maintheorem} {\it (i)}$\implies${\it (ii)} and lower bound for $\|\tilde\varphi\|_S$]
	If $\tilde\varphi$ is a Schur multiplier on $\X$ then, according to Lemma~\ref{staralg}, there exists a bounded linear functional $f_\varphi$ on $\cstar(U)$ satisfying~\eqref{diamond} and~\eqref{wr}. Now use Proposition~\ref{functional} to find a complex Borel measure $\mu$ on $\T$ and a trace class operator $T'$ on $\ell^2(\N_0)$ such that
	\begin{equation}
		\label{pfd1}
		f_\varphi(U^m(U^*)^n)=\int_\T z^{m-n}\dd\mu(z)+\trace(S^m(S^*)^nT')\qquad(m,n\in\N_0)
	\end{equation}
	and
	\begin{equation}
		\label{pfdd1}
		\|f_\varphi\|=\|\mu\|+\|T'\|_1.
	\end{equation}
	Put $T=\big(1-\tfrac{1}{q}\big)^{-1}\big(I-\tfrac{\tau}{q}\big)T'$ and recall that
	\begin{equation*}
		\trace(S^i(S^*)^jT)=\trace(S_{i,j}T')\qquad(i,j\in\N_0),
	\end{equation*}
	according to Lemma~\ref{HH}. Using this,~\eqref{diamond} and~\eqref{pfd1} we find that
	\begin{equation}
		\label{pfd2}
		\dot\varphi(m+n)=\int_\T z^{m-n}\dd\mu(z)+\trace(S^m(S^*)^nT)\qquad(m,n\in\N_0).
	\end{equation}
	Using~\eqref{wr} and~\eqref{pfdd1} we find that
	\begin{equation}
		\label{pfdd2}
		\|\tilde\varphi\|_S\geq\|\mu\|+\|T'\|_1.
	\end{equation}
	Fix an arbitrary $k\in\Z$ and use~\eqref{pfd2} to see that
	\begin{equation*}
		\dot\varphi(2n+k)=\int_\T z^k\dd\mu(z)+\trace(S^{n+k}(S^*)^nT)\qquad(n+k,n\in\N_0).
	\end{equation*}
	For $n+k,n\in\N_0$ put $n_0=\max\seto{0,-k}$ and note that $n_0\leq n$.
	Observe that
	\begin{equation*}
		\trace(S^{n+k}(S^*)^nT)=\sum_{l=n}^\infty t_{l,l+k}\qquad(n+k,n\in\N_0),
	\end{equation*}
	when $T=(t_{i,j})_{i,j\in\N_0}$. Also,
	\begin{equation*}
		\lim_{n\to\infty}\sum_{l=n}^\infty t_{l,l+k}=0
	\end{equation*}
	since
	\begin{equation*}
		\sum_{l=n_0}^\infty |t_{l,l+k}|<\infty,
	\end{equation*}
	which follows from the fact that $S^{n_0+k}(S^*)^{n_0}T$ is of trace class. Hence
	\begin{equation*}
		\lim_{n\to\infty}\trace(S^{n+k}(S^*)^nT)=0
	\end{equation*}
	so we conclude that
	\begin{equation*}
		\lim_{n\to\infty}\dot\varphi(2n+k)=\int_\T z^k\dd\mu(z),
	\end{equation*}
	and therefore
	\begin{equation*}
		\int_\T z^k\dd\mu(z)=\int_\T z^{k+2}\dd\mu(z)\qquad(k\in\Z).
	\end{equation*}
	Hence, there exists $a,b\in\C$ such that
	\begin{equation*}
		\int_\T z^k\dd\mu(z)=\left\{
		\begin{array}{lll}
			a & \mbox{if} & k\mbox{ is even}\\
			b & \mbox{if} & k\mbox{ is odd}
		\end{array}
		\right..
	\end{equation*}
	Put $c_\pm=\tfrac{1}{2}(a\pm b)$ and let $\nu$ be the complex measure on $\T$ given by
	\begin{equation*}
		\nu=c_+\delta_{+1}+c_-\delta_{-1},
	\end{equation*}
	where $\delta_1$ (respectively $\delta_{-1}$) is the Dirac measure at $1$ (respectively $-1$). Then
	\begin{equation*}
		\int_\T z^k\dd\nu(z)=c_++(-1)^k c_-=\int_\T z^k\dd\mu(z)\qquad(k\in\Z).
	\end{equation*}
	Hence $\mu=\nu$ and we have according to~\eqref{pfd2} and~\eqref{pfdd2}
	\begin{equation}
		\label{pfd3}
		\dot\varphi(m+n)=c_++c_-(-1)^{m+n}+\trace(S^m(S^*)^nT)\qquad(m,n\in\N_0)
	\end{equation}
	and
	\begin{equation}
		\label{pfdd3}
		\|\tilde\varphi\|_S\geq|c_+|+|c_-|+\|T'\|_1.
	\end{equation}
	This finishes the second part of the proof of Theorem~\ref{maintheorem}, since
	\begin{equation*}
		t_{m,n}=\trace(S^n(S^*)^mT)-\trace(S^{n+1}(S^*)^{m+1}T)=\dot\varphi(m+n)-\dot\varphi(m+n+2)=h_{m,n}
	\end{equation*}
	for all $m,n\in\N_0$.
\end{proof}
This concludes the final step of the proof of Theorem~\ref{maintheorem}.

\bigskip\noindent
In the rest of this section we let $\X$ denote (the vertices of) a homogeneous tree of infinite degree, and consider the pair $(\X,x_0)$ where $x_0$ is a distinguished vertex in $\X$. For $2\leq q<\infty$ let $\X_q$ be a homogeneous subtree of degree $q+1$ containing $x_0$ (besides from $x_0$, we do not care which vertices are removed, since we will exclusively look at radial functions anyway). Obviously, there is a bijective correspondence between radial functions on $\X$ and radial functions on $\X_q$, and given $\dot\varphi:\N_0\to\C$ we will consider both $\varphi:\X\to\C$ and the restriction $\varphi|_{\X_q}:\X_q\to\C$ of $\varphi$ to $\X_q$.
\begin{lemma}
	\label{HH2}
	Let $T,T'\in\linbeg(\ell^2(\N_0))$ be related by
	\begin{equation*}
		T'=\big(1-\tfrac{1}{q}\big)\big(I-\tfrac{\tau}{q}\big)^{-1}T.
	\end{equation*}
	Assume that one, and hence both, matrices are of trace class, then
	\begin{equation*}
		\frac{q-1}{q+1}\|T\|_1\leq\|T'\|_1\leq\|T\|_1.
	\end{equation*}
\end{lemma}
\begin{proof}
	This follows using $\|\big(I-\frac{\tau}{q}\big)^{-1}\|\leq\big(1-\tfrac{1}{q}\big)^{-1}$ and $\|I-\frac{\tau}{q}\|\leq1+\tfrac{1}{q}$, where both operators are considered as elements of $\linbeg(\traceclass(\ell^2(\N_0)))$.
\end{proof}
\begin{lemma}
	\label{Smninf}
	For $x,y\in\X$ we have that
	\begin{equation*}
		(S^m(S^*)^n)_{i,j}=\ip{\delta_{\cc^j(y)}}{\delta_{\cc^i(x)}}\qquad(i,j\in\N_0),
	\end{equation*}
	when $m,n\in\N_0$ are chosen as in Remark~\ref{xymn}.
\end{lemma}
\begin{proof}
	This is an easy verification.
\end{proof}
\begin{maintheorem}
	\label{maintheoreminf}
	Let $(\X,x_0)$ be a homogeneous tree of infinite degree with a distinguished vertex $x_0\in\X$. Let $\varphi:\X\to\C$ be a radial function and let $\dot\varphi:\N_0\to\C$ and $\tilde\varphi:\X\times\X\to\C$ be the corresponding functions as in Proposition~\ref{radial}. Finally, let $H=(h_{i,j})_{i,j\in\N_0}$ be the Hankel matrix given by $h_{i,j}=\dot\varphi(i+j)-\dot\varphi(i+j+2)$ for $i,j\in\N_0$. Then the following are equivalent:
	\begin{itemize}
		\item [(i)]$\tilde\varphi$ is a Schur multiplier.
		\item [(ii)]$H$ is of trace class.
	\end{itemize}
	If these two equivalent conditions are satisfied, then there exists unique constants $c_\pm\in\C$ and a unique $\dot\psi:\N_0\to\C$ such that
	\begin{equation*}
		\dot\varphi(n)=c_++c_-(-1)^n+\dot\psi(n)\qquad(n\in\N_0)
	\end{equation*}
	and
	\begin{equation*}
		\lim_{n\to\infty}\dot\psi(n)=0.
	\end{equation*}
	Moreover,
	\begin{equation*}
		\|\tilde\varphi\|_S=|c_+|+|c_-|+\|H\|_1.
	\end{equation*}
\end{maintheorem}
\begin{proof}
	Let $\varphi|_{\X_q}$ be the restriction of $\varphi$ to $\X_q$ for $2\leq q<\infty$, where $\X_q$ is a homogeneous subtree of $\X$ of degree $q+1$ containing $x_0$. From Proposition~\ref{Grothendieck} it is easily seen that if $\tilde\varphi$ is a Schur multiplier, then the restriction $\tilde\varphi|_{\X_q\times\X_q}$ ($2\leq q<\infty$) is also a Schur multiplier, and $\|\tilde\varphi|_{\X_q\times\X_q}\|_S\leq\|\tilde\varphi\|_S$. Using this together with Theorem~\ref{maintheorem} and Lemma~\ref{HH2} it is easy to see that we are left with the task of proving {\it (ii)}$\implies${\it (i)} and the upper bound for $\|\tilde\varphi\|_S$. But this basically consists of taking the corresponding part of the proof of Theorem~\ref{maintheorem} and deleting all the primes, so we only provide a sketchy proof of this.
	
	Assume that the Hankel matrix $H=(h_{i,j})_{i,j\in\N_0}$ is of trace class, define $c_\pm$ and $\dot\psi$ as in the first part of the proof of Theorem~\ref{maintheorem} and note that
	\begin{equation}
		\label{psitraceinf}
		\dot\psi(i+j)=\trace(S^i(S^*)^jH)\qquad(i,j\in\N_0).
	\end{equation}
	Since $H$ is of trace class there exists sequences $(\xi^{(k)})_{k\in\N_0},(\eta^{(k)})_{k\in\N_0}$ in $\ell^2(\N_0)$ such that
	\begin{equation*}
		H=\sum_{k=0}^\infty\xi^{(k)}\odot\eta^{(k)}\quad\mbox{and}\quad\|H\|_1=\sum_{k=0}^\infty\|\xi^{(k)}\|_2\|\eta^{(k)}\|_2,
	\end{equation*}
	and therefore
	\begin{equation}
		\label{8.5x2}
		h_{i,j}=\sum_{k=0}^\infty\xi_i^{(k)}\bar\eta_j^{(k)}\qquad(i,j\in\N_0).
	\end{equation}
	For each $k\in\N_0$ we define $P_k,Q_k:X\to\ell^2(\X)$ by
	\begin{equation*}
		P_k(x)=\sum_{i=0}^\infty\xi_i^{(k)}\delta_{\cc^i(x)}\quad\mbox{and}\quad Q_k(y)=\sum_{j=0}^\infty\eta_j^{(k)}\delta_{\cc^j(y)}\qquad(x,y\in\X),
	\end{equation*}
	and note that
	\begin{equation*}
		\sum_{k=0}^\infty\|P_k\|_\infty\|Q_k\|_\infty=\|H\|_1.
	\end{equation*}
	Now verify that
	\begin{equation*}
		\sum_{k=0}^\infty\ip{P_k(x)}{Q_k(y)}=\sum_{k,i,j=0}^\infty\ip{\delta_{\cc^i(x)}}{\delta_{\cc^j(y)}}\xi_i^{(k)}\bar\eta_j^{(k)}=\sum_{i,j=0}^\infty\ip{\delta_{\cc^i(x)}}{\delta_{\cc^j(y)}} h_{i,j}
	\end{equation*}
	for all $x,y\in\X$. Momentarily fix $x,y\in\X$ and choose $m,n\in\N_0$ according to Remark~\ref{xymn}. Then $m+n=\dd(x,y)$ and by Lemma~\ref{Smninf}
	\begin{equation*}
		(S^m(S^*)^n)_{j,i}=\ip{\delta_{\cc^i(x)}}{\delta_{\cc^j(y)}}\qquad(i,j\in\N_0).
	\end{equation*}
	Using~\eqref{psitraceinf} it follows that
	\begin{eqnarray*}
		\sum_{k=0}^\infty\ip{P_k(x)}{Q_k(y)} & = & \sum_{i,j=0}^\infty (S^m(S^*)^n)_{j,i}h_{i,j}\\
		& = & \trace(S^m(S^*)^nH)\\
		& = & \dot\psi(m+n)\\
		& = & \tilde\psi(x,y).
	\end{eqnarray*}
	Since $x,y\in\X$ were arbitrary we have that
	\begin{equation*}
		\tilde\varphi(x,y)=\dot\varphi(\dd(x,y))=c_++c_-(-1)^{\dd(x,y)}+\sum_{k=0}^\infty\ip{P_k(x)}{Q_k(y)}\qquad(x,y\in\X).
	\end{equation*}
	We conclude that $\tilde\varphi$ is a Schur multiplier with
	\begin{equation*}
		\|\tilde\varphi\|_S\leq |c_+|+|c_-|+\sum_{k=0}^\infty\|P_k\|_\infty\|Q_k\|_\infty=|c_+|+|c_-|+\|H\|_1.
	\end{equation*}
\end{proof}
\begin{corollary}
	\label{sameDifference}
	Let $(\X,x_0)$ be a homogeneous tree of infinite degree with distinguished vertex $x_0$. Choose as before for each integer $2\leq q<\infty$ a homogeneous subtree $\X_q\subseteq\X$ of degree $q$ with $x_0\in\X_q$. Let $\dot\varphi:\N_0\to\C$ be given and define $\tilde\varphi:\X\times\X\to\C$ as in Proposition~\ref{radial}. Then $\tilde\varphi$ is a Schur multiplier if and only if $\tilde\varphi|_{\X_q\times\X_q}$ is a Schur multiplier. Moreover,
	\begin{equation*}
		\frac{q-1}{q+1}\|\tilde\varphi\|_S\leq\|\tilde\varphi|_{\X_q\times\X_q}\|_S\leq\|\tilde\varphi\|_S.
	\end{equation*}
\end{corollary}
\begin{proof}
	This follows from Theorem~\ref{maintheorem}, Theorem~\ref{maintheoreminf} and Lemma~\ref{HH2}.
\end{proof}

	\section{Spherical functions on homogeneous trees}
\label{sphfct}
As in section~\ref{homtree}, we begin by considering a pair $(\X,x_0)$, where $\X$ is a homogeneous tree of degree $q+1$ for $2\leq q<\infty$ and $x_0$ is a distinguished vertex in $\X$. Later on we will also consider the case when $\X$ has infinite degree. We give only a brief introduction to spherical functions on homogeneous trees of finite degree---the reader is referred to~\cite{FTN:HarmonicAnalysisAndRepresentationTheoryForGroupsActingOnHhomogeneousTrees} for a more thorough exposition.

If $\varphi$ is a (complex valued) function on $\X$ we let (for any $x\in\X$) $L\varphi(x)$ denote the average value of $\varphi$ over the vertices which share an edge with $x$. The operator $L$ is called the \emph{Laplace operator} on $\X$. Following~\cite[Chapter~II, Definition~2.2]{FTN:HarmonicAnalysisAndRepresentationTheoryForGroupsActingOnHhomogeneousTrees}, we have:
\begin{definition}
	\label{spherical}
	Let $(\X,x_0)$ be a homogeneous tree of degree $q+1$ ($2\leq q<\infty$) with a distinguished vertex $x_0$. A radial function $\varphi:\X\to\C$ is called a \emph{spherical function} (on $(\X,x_0)$) if it satisfies
	\begin{itemize}
		\item [(i)]$\varphi(x_0)=1$
		\item [(ii)]$L\varphi=\eval\varphi$ for some $\eval\in\C$.
	\end{itemize}
\end{definition}
The number $\eval$ is called the \emph{eigenvalue} corresponding to the spherical function $\varphi$. Since $\varphi$ is radial it has the form
\begin{equation*}
	\varphi(x)=\dot\varphi(\dd(x,x_0))
\end{equation*}
for some $\dot\varphi:\N_0\to\C$ (cf.~Proposition~\ref{radial}). One can rewrite {\rm (i)} and {\rm (ii)} as
\begin{equation}
	\label{sph1}
	\dot\varphi(0)=1
\end{equation}
\begin{equation}
	\label{sph2}
	\dot\varphi(1)=\eval
\end{equation}
\begin{equation}
	\label{sph3}
	\dot\varphi(n+1)=\eval(1+\tfrac{1}{q})\dot\varphi(n)-\tfrac{1}{q}\dot\varphi(n-1)\qquad(n\in\N),
\end{equation}
cf.~\cite[page~42--43]{FTN:HarmonicAnalysisAndRepresentationTheoryForGroupsActingOnHhomogeneousTrees}. In particular, a spherical function is uniquely determined by its eigenvalue $\eval$. Despite of this, one does not label the spherical function by their eigenvalue---this is due to tradition and the fact that calculations seem to work out most easily using another indexation. For $z\in\C$ define the function $\varphi_z:\X\to\C$ by\footnote{The given expression (considered as a function of $z$ for fixed $x$) has a removable singularity when $q^{-z}=q^{z-1}$, or equivalently, when $z=\tfrac{1}{2}+i\tfrac{k\pi}{\ln(q)}$ for $k\in\Z$.}
\begin{equation}
	\label{varphiz2}
	\varphi_z(x)=f(z)h_z(x)+f(1-z)h_{1-z}(x)\qquad(x\in\X),
\end{equation}
where
\begin{equation}
	\label{hz}
	h_z(x)=q^{-z\dd(x,x_0)}\quad\mbox{and}\quad f(z)=(q+1)^{-1}\frac{q^{1-z}-q^{z-1}}{q^{-z}-q^{z-1}}\qquad(x\in\X).
\end{equation}
According to \cite{FTN:HarmonicAnalysisAndRepresentationTheoryForGroupsActingOnHhomogeneousTrees}, $\varphi_z$ is a spherical function on $(\X,x_0)$ with eigenvalue
\begin{equation}
	\label{sz}
	\eval_z=(1+\tfrac{1}{q})^{-1}(q^{-z}+q^{z-1})\qquad(z\in\C).
\end{equation}
The complex function
\begin{equation*}
	\cosh(z)=\tfrac{1}{2}(e^z+e^{-z})\qquad(z\in\C)
\end{equation*}
maps $\C$ onto $\C$ and $\cosh(z)=\cosh(z')$ if and only if $z'\in z+2\pi i\Z$ or $z'\in -z+2\pi i\Z$. Since
\begin{equation}
	\label{sz2}
	\eval_z=\frac{2\sqrt q}{q+1}\cosh(\ln(q)(z-\tfrac{1}{2}))\qquad(z\in\C)
\end{equation}
it follows that every spherical function $\varphi:\X\to\C$ is of the form $\varphi_z$ for a $z\in\C$ and that $\varphi_z=\varphi_{z'}$ if and only if
\begin{equation}
	\label{zzp}
	z'\in z+i\frac{2\pi}{\ln(q)}\Z\quad\mbox{or}\quad z'\in 1-z+i\frac{2\pi}{\ln(q)}\Z.
\end{equation}

From the definition of $\varphi_z$ it is easily seen that $\varphi_z$ is bounded if and only if $0\leq\Re(z)\leq1$. Since a Schur multiplier must be bounded, $\tilde\varphi_z$ is not a Schur multiplier if $\Re(z)<0$ or $\Re(z)>1$. The following theorem states which of the remaining $z$'s give rise to Schur multipliers and explicitly specifies the corresponding Schur norm.
\begin{theorem}
	\label{sphSchurNorm}
	Let $(\X,x_0)$ be a homogeneous tree of degree $q+1$ ($2\leq q<\infty$) with a distinguished vertex $x_0\in\X$. For $z\in\C$ let $\varphi_z:\X\to\C$ be the spherical function given by~\eqref{varphiz2} and let $\tilde\varphi_z:\X\times\X\to\C$ be the corresponding function as in Proposition~\ref{radial}. Then $\tilde\varphi_z$ is a Schur multiplier if and only if $z$ is in the set
	\begin{equation*}
		\setw{z\in\C }{ 0<\Re(z)<1 }\bigcup\setw{i\frac{k\pi}{\ln(q)} }{ k\in\Z }\bigcup\setw{1-i\frac{k\pi}{\ln(q)} }{ k\in\Z }.
	\end{equation*}
	The corresponding norm is given by
	\begin{equation*}
		\|\tilde\varphi_z\|_S=\frac{(1-\tfrac{1}{q})^2|1-q^{-2z}||1-q^{2z-2}|}{(1-q^{-2\Re(z)})(1-q^{2\Re(z)-2})|1-q^{2i\Im(z)-1}|^2}\qquad(0<\Re(z)<1)
	\end{equation*}
	and
	\begin{equation*}
		\|\tilde\varphi_z\|_S=1\qquad(z\in\setw{i\frac{k\pi}{\ln(q)} }{ k\in\Z }\bigcup\setw{1-i\frac{k\pi}{\ln(q)} }{ k\in\Z }).
	\end{equation*}
\end{theorem}
\begin{proof}
	Since $\eval_z=\eval_{1-z}$ for all $z\in\C$ it is enough to consider $z\in\C$ for which $\Re(z)\leq\tfrac{1}{2}$. The case $\Re(z)<0$ has already been considered above, with the conclusion that $\tilde\varphi_z$ is not bounded and hence not a Schur multiplier. Next we treat the case $\Re(z)=0$, so write $z=it$ with $t\in\R$. Theorem~\ref{maintheorem} tells us that if $\tilde\varphi_{it}$ is a Schur multiplier, then $\lim_{n\to\infty}\dot\varphi_{it}(2n)$ must exist (and equal $c_++c_-$). But it is easily seen from~\eqref{varphiz2} that
	\begin{equation*}
		\lim_{n\to\infty}(\dot\varphi_{it}(2n)-f(it)q^{-2n it})=0
	\end{equation*}
	and from~\eqref{hz} that $f(it)\neq0$.
	Hence, the only $t\in\R$ where $\lim_{n\to\infty}\dot\varphi_{it}(2n)$ exists is $t=\frac{k\pi}{\ln(q)}$ for $k\in\Z$.
	From~\eqref{varphiz2} and~\eqref{hz} we get that for $n\in\N_0$,
	\begin{equation*}
		\dot\varphi_{i\frac{k\pi}{\ln(q)}}(n)=\left\{
		\begin{array}{lll}
			1 & \mbox{if} & k\mbox{ is even}\\
			(-1)^n & \mbox{if} & k\mbox{ is odd}
		\end{array}
		\right..
	\end{equation*}
	Hence for $x,y\in\X$,
	\begin{equation*}
		\tilde\varphi_{i\frac{k\pi}{\ln(q)}}(x,y)=\left\{
		\begin{array}{lll}
			1 & \mbox{if} & k\mbox{ is even}\\
			(-1)^{\dd(x,y)} & \mbox{if} & k\mbox{ is odd}
		\end{array}
		\right..
	\end{equation*}
	Since $(-1)^{\dd(x,y)}=(-1)^{\dd(x,x_0)}(-1)^{\dd(y,x_0)}$ it follows that $\tilde\varphi_{i\frac{k\pi}{\ln(q)}}$ is a Schur multiplier with Schur norm $1$ for all $k\in\Z$.
	
	We are left with the case $0<\Re(z)\leq\tfrac{1}{2}$. In order to show that $\tilde\varphi_z$ is a Schur multiplier it is enough, according to Theorem~\ref{maintheorem} and Lemma~\ref{HH}, to verify that $H'$ given by $H'=\big(I-\tfrac{\tau}{q}\big)^{-1}H$, where $H$ is the Hankel matrix $H=(h_{i,j})_{i,j\in\N_0}$ with entries $h_{i,j}=\dot\varphi_z(i+j)-\dot\varphi_z(i+j+2)$, is of trace class. To find the corresponding Schur norm we must compute the actual trace class norm of $H'$ and also find $c_\pm$ from Theorem~\ref{maintheorem}. But since $\lim_{n\to\infty}\dot\varphi(n)=0$ when $0<\Re(z)\leq\tfrac{1}{2}$ we conclude that $c_\pm=0$, so we are left with just calculating $\|H'\|_1$.
	
	Fix a $z$ with $0<\Re(z)\leq\tfrac{1}{2}$, put $a=q^{-z}$ and $b=q^{z-1}$ and note that $|a|,|b|<1$. From~\eqref{varphiz2} we find, after some manipulations, that\footnote{Here we again run into the problem of a singularity whenever $a=b$, but all calculations can be done by replacing $\frac{a^{n+1}-b^{n+1}}{a-b}$ with $(n+1)a^n$ (for $n\in\N_0$) when this happens.}
	\begin{equation*}
		h_{i,j}=(1+\tfrac{1}{q})^{-1}(1-a^2)(1-b^2)\frac{a^{i+j+1}-b^{i+j+1}}{a-b}\qquad(i,j\in\N_0).
	\end{equation*}
	From this it is quite easy to verify that
	\begin{equation*}
		h'_{i,j}=(1-\tfrac{1}{q})(1+\tfrac{1}{q})^{-1}(1-a^2)(1-b^2)\frac{a^{i+1}-b^{i+1}}{a-b}\frac{a^{j+1}-b^{j+1}}{a-b}\qquad(i,j\in\N_0),
	\end{equation*}
	where $H'=(h'_{i,j})_{i,j\in\N_0}$. Now observe that
	\begin{equation*}
		H'=\alpha\xi\odot\eta,
	\end{equation*}
	where
	\begin{equation*}
		\alpha=(1-\tfrac{1}{q})(1+\tfrac{1}{q})^{-1}(1-a^2)(1-b^2)
	\end{equation*}
	and $\xi=(\xi_n)_{n\in\N_0},\eta=(\eta_n)_{n\in\N_0}\in\ell^2(\N_0)$ are given by
	\begin{equation*}
		\xi_n=\frac{a^{n+1}-b^{n+1}}{a-b}=\bar\eta_n\qquad(n\in\N_0).
	\end{equation*}
	From~\eqref{1x} we conclude that
	\begin{equation*}
		\|H'\|_1=|\alpha|\|\xi\|_2\|\eta\|_2=|\alpha|\sum_{n=0}^\infty\left|\frac{a^{n+1}-b^{n+1}}{a-b}\right|^2.
	\end{equation*}
	Simple, but tedious, computations shows that\footnote{This formula also holds for $a=b$ and $c=d$ when $\frac{a^{n+1}-b^{n+1}}{a-b}$ and $\frac{c^{n+1}-d^{n+1}}{c-d}$ are replaced by $(n+1)a^n$ and $(n+1)c^n$, respectively.}
	\begin{equation*}
		\sum_{n=0}^\infty\frac{a^{n+1}-b^{n+1}}{a-b}\frac{c^{n+1}-d^{n+1}}{c-d}=\frac{1-a b c d}{(1-a c)(1-b d)(1-a d)(1-b c)},
	\end{equation*}
	when $a,b,c,d\in\C$ satisfy $|a|,|b|,|c|,|d|<1$. Using this result with $c=\bar a$ and $d=\bar b$ we find that
	\begin{equation*}
		\|H'\|_1=\frac{(1-\tfrac{1}{q})^2|1-a^2||1-b^2|}{(1-a\bar a)(1-b\bar b)(1-a\bar b)(1-b\bar a)},
	\end{equation*}
	which is easily seen to finish the proof.
\end{proof}
We now reformulate the above theorem in terms of eigenvalues.
\begin{maintheorem}
	\label{sphSchurNormEigenvalue}
	Let $(\X,x_0)$ be a homogeneous tree of degree $q+1$ ($2\leq q<\infty$) with a distinguished vertex $x_0\in\X$. Let $\varphi:\X\to\C$ be a spherical function and let $\tilde\varphi:\X\times\X\to\C$ be the corresponding function as in Proposition~\ref{radial}. Then $\tilde\varphi$ is a Schur multiplier if and only if the eigenvalue $\eval$ corresponding to $\varphi$ is in the set
	\begin{equation*}
		\setw{\eval\in\C }{ \Re(\eval)^2+\big(\tfrac{q+1}{q-1}\big)^2\Im(\eval)^2<1 }\bigcup\seto{\pm1}.
	\end{equation*}
	The corresponding Schur norm is given by
	\begin{equation*}
		\|\tilde\varphi\|_S=\frac{|1-\eval^2|}{1-\Re(\eval)^2-\left(\frac{q+1}{q-1}\right)^2\Im(\eval)^2}\qquad(\Re(\eval)^2+\big(\tfrac{q+1}{q-1}\big)^2\Im(\eval)^2<1)
	\end{equation*}
	and
	\begin{equation*}
		\|\tilde\varphi\|_S=1\qquad(\eval=\pm1).
	\end{equation*}
\end{maintheorem}
\begin{proof}
	Observe that $z\mapsto\eval_z$ maps $\setw{z\in\C }{ 0<\Re(z)<1 }$ onto $\setw{\eval\in\C }{ \Re(\eval)^2+\big(\tfrac{q+1}{q-1}\big)^2\Im(\eval)^2<1 }$ and maps $\setw{i\frac{k\pi}{\ln(q)} }{ k\in\Z }\bigcup\setw{1-i\frac{k\pi}{\ln(q)} }{ k\in\Z }$ onto $\seto{\pm1}$. The Corollary now follows readily from Theorem~\ref{sphSchurNorm} once we have shown that
	\begin{eqnarray}
		& & \frac{(1-\tfrac{1}{q})^2|1-q^{-2z}||1-q^{2z-2}|}{(1-q^{-2\Re(z)})(1-q^{2\Re(z)-2})|1-q^{2i\Im(z)-1}|^2}\label{*N}\\
		& = & \frac{|1-\eval_z^2|}{1-\Re(\eval_z)^2-\left(\frac{q+1}{q-1}\right)^2\Im(\eval_z)^2}\label{**N}
	\end{eqnarray}
	for $z\in\C$ with $0<\Re(z)<1$.
	
	Fix a $z$ with $0<\Re(z)<1$ and (as in the proof of Theorem~\ref{sphSchurNorm}) put $a=q^{-z}$ and $b=q^{z-1}$ and note that $|a|,|b|<1$. In order to eliminate $a$ and $b$ (in exchange for expressions involving the eigenvalue) we observe the following relations
	\begin{equation}
		\label{abmu}
		a b=\tfrac{1}{q},\quad a+b=(1+\tfrac{1}{q})\eval_z\quad\mbox{and}\quad a^2+b^2=(1+\tfrac{1}{q})^2\eval_z^2-\tfrac{2}{q},
	\end{equation}
	where the first two relations are obvious, and the third one follows from the first two through $a^2+b^2=(a+b)^2-2a b$. First, look at the nominator in~\eqref{*N} and use~\eqref{abmu} (and a little work) to arrive at
	\begin{equation*}
		|1-q^{-2z}||1-q^{2z-2}|=|1-a^2-b^2+a^2b^2|=(1+\tfrac{1}{q})^2|1-\eval_z^2|.
	\end{equation*}
	Now, look at the denominator in~\eqref{*N} and use~\eqref{abmu} to find that
	\begin{eqnarray*}
		&&(1-q^{-2\Re(z)})(1-q^{2\Re(z)-2})|1-q^{2i\Im(z)-1}|^2\\
		& = & (1-\bar a a)(1-\bar b b)(1-a\bar b)(1-\bar a b)\\
		& = & (1+\tfrac{1}{q})^2(1-\tfrac{1}{q})^2+\tfrac{2}{q}(1+\tfrac{1}{q})^2\Re(\eval_z^2)-(1+\tfrac{1}{q^2})(1+\tfrac{1}{q})^2\bar\eval_z\eval_z.
	\end{eqnarray*}
	Use
	\begin{equation*}
		\Re(\eval_z^2)=\Re(\eval_z)^2-\Im(\eval_z)^2\quad\mbox{and}\quad \bar\eval_z\eval_z=\Re(\eval_z)^2+\Im(\eval_z)^2
	\end{equation*}
	to continue the above calculation and arrive at
	\begin{eqnarray*}
		&&(1-q^{-2\Re(z)})(1-q^{2\Re(z)-2})|1-q^{2i\Im(z)-1}|^2\\
		& = & (1+\tfrac{1}{q})^2(1-\tfrac{1}{q})^2(1-\Re(\eval_z)^2-\big(\tfrac{q+1}{q-1}\big)^2\Im(\eval_z)^2).
	\end{eqnarray*}
	Putting the calculations for the nominator together with the calculations for the denominator we easily arrive at~\eqref{**N}.
\end{proof}

\bigskip\noindent
To define the spherical functions on $(\X,x_0)$ when $\X$ has infinite degree, we generalize the \emph{Laplace operator} $L$ to the case of infinite degree. However, this only makes sense for radial functions. If $\varphi$ is a radial function on $\X$ let $L\varphi$ denote the radial function on $\X$ for which
\begin{equation*}
	L\varphi(x)=\dot\varphi(\dd(x,x_0)+1)\qquad(x\in\X),
\end{equation*}
where $\dot\varphi$ is connected to $\varphi$ as in Proposition~\ref{radial}.
\begin{definition}
	\label{sphericalinf}
	Let $(\X,x_0)$ be a homogeneous tree of infinite degree with a distinguished vertex $x_0$. A radial function $\varphi:\X\to\C$ is called a \emph{spherical function} (on $(\X,x_0)$) if it satisfies
	\begin{itemize}
		\item [(i)]$\varphi(x_0)=1$
		\item [(ii)]$L\varphi=\eval\varphi$ for some $\eval\in\C$.
	\end{itemize}
\end{definition}
The number $\eval$ is again called the \emph{eigenvalue} corresponding to the spherical function $\varphi$. Using that
\begin{equation*}
	\varphi(x)=\dot\varphi(\dd(x,x_0))
\end{equation*}
for a function $\dot\varphi:\N_0\to\C$ one can now rewrite {\rm (i)} and {\rm (ii)} as
\begin{equation}
	\label{sph1inf}
	\dot\varphi(0)=1
\end{equation}
\begin{equation}
	\label{sph2inf}
	\dot\varphi(1)=\eval
\end{equation}
\begin{equation}
	\label{sph3inf}
	\dot\varphi(n+1)=\eval\dot\varphi(n)\qquad(n\in\N).
\end{equation}
Note that~\eqref{sph1inf}--\eqref{sph3inf} can be considered as the limits for $q$ going to infinity of~\eqref{sph1}--\eqref{sph3}. Obviously, the spherical function corresponding to the eigenvalue $\eval\in\C$ is given by
\begin{equation}
	\label{varphiz3}
	\dot\varphi(n)=\eval^n\qquad(n\in\N_0).
\end{equation}
As a corollary to Theorem~\ref{maintheoreminf} we get the following theorem.
\begin{theorem}
	\label{sphSchurNorminf}
	Let $(\X,x_0)$ be a homogeneous tree of infinite degree with a distinguished vertex $x_0\in\X$. Let $\varphi:\X\to\C$ be a spherical function and let $\tilde\varphi:\X\times\X\to\C$ be the corresponding function as in Proposition~\ref{radial}. Then $\tilde\varphi$ is a Schur multiplier if and only if the eigenvalue $\eval$ corresponding to $\varphi$ is in the set
	\begin{equation*}
		\setw{\eval\in\C }{ |\eval|<1 }\bigcup\seto{\pm1}.
	\end{equation*}
	The corresponding norm is given by
	\begin{equation*}
		\|\tilde\varphi\|_S=\frac{|1-\eval^2|}{1-|\eval|^2}\qquad(|\eval|<1)
	\end{equation*}
	and
	\begin{equation*}
		\|\tilde\varphi\|_S=1\qquad(\eval=\pm1).
	\end{equation*}
\end{theorem}
\begin{proof}
	If $|\eval|>1$ then $\varphi$ is unbounded, and therefore $\tilde\varphi$ is not a Schur multiplier. If $|\eval|=1$ then $\lim_{n\to\infty}\dot\varphi_{it}(2n)$ only exits for $\eval=\pm1$, so these are the only eigenvalues with length one for which $\tilde\varphi$ can be a Schur multiplier. But it is easy to see that if $\varphi$ is the spherical function with eigenvalue $\eval=\pm1$, then $\tilde\varphi$ is a Schur multiplier with Schur norm $1$.
	
	Finally, we are left with the case $|\eval|<1$. Fix $\eval\in\C$ with $|\eval|<1$ and let $\varphi$ be the spherical function with eigenvalue $\eval$. In order to show that $\tilde\varphi$ is a Schur multiplier it is enough, according to Theorem~\ref{maintheoreminf}, to verify that the Hankel matrix $H=(h_{i,j})_{i,j\in\N_0}$ given by $h_{i,j}=\dot\varphi_z(i+j)-\dot\varphi_z(i+j+2)$ is of trace class. To find the corresponding Schur norm we must compute the actual trace class norm of $H$ and also find $c_\pm$ from Theorem~\ref{maintheoreminf}. But since $\lim_{n\to\infty}\dot\varphi(n)=0$ we conclude that $c_\pm=0$, so we are left with just calculating $\|H\|_1$. Start by noting that
	\begin{equation*}
		h_{i,j}=\eval^{i+j}-\eval^{i+j+2}=(1-\eval^2)\eval^i\eval^j\qquad(i,j\in\N_0).
	\end{equation*}
	Now observe that
	\begin{equation*}
		H=\alpha\xi\odot\eta,
	\end{equation*}
	where
	\begin{equation*}
		\alpha=(1-\eval^2)
	\end{equation*}
	and $\xi=(\xi_n)_{n\in\N_0},\eta=(\eta_n)_{n\in\N_0}\in\ell^2(\N_0)$ are given by
	\begin{equation*}
		\xi_n=\eval^n=\bar\eta_n\qquad(n\in\N_0).
	\end{equation*}
	From~\eqref{1x} we conclude that
	\begin{equation*}
		\|H\|_1=|\alpha|\|\xi\|_2\|\eta\|_2=|1-\eval^2|\sum_{n=0}^\infty|\eval^n|^2=\frac{|1-\eval^2|}{1-|\eval|^2}.
	\end{equation*}
\end{proof}

	\section{Integral representations of radial Schur multipliers}
\label{integral}
Let $(\kk_n)_{n=1}^\infty$ be a bounded sequence of complex numbers. Consider the Hankel matrix $H=(h_{i,j})_{i,j\in\N_0}$ given by
\begin{equation*}
	h_{i,j}=\kk_{i+j}\qquad(i,j\in\N_0)
\end{equation*}
and the analytic function $f$ on the open unit disc $\D$ in $\C$ given by
\begin{equation*}
	f(z)=\sum_{n=0}^\infty\kk_n z^n\qquad(z\in\D).
\end{equation*}
By a theorem of Peller~\cite[Theorem~1']{Pel:HankelOperatorsOfClassSpAndTheirApplications(rationalApproximationGaussianProcessesTheProblemOfMajorizationOfOperators)}, $H$ is of trace class if and only if the second derivative $f''$ of $f$ is in $L^1(\D)$, i.e.,
\begin{equation*}
	\|f''\|_1=\frac{1}{\pi}\int_\D|f''(z)|\dd z_1\dd z_2<\infty,
\end{equation*}
where $z_1=\Re(z)$ and $z_2=\Im(z)$. The following theorem is a slight variation of Peller's result, which tells that $H$ is of trace class if and only if the function
\begin{equation*}
	g(z)=\frac{\dd^2}{\dd z^2}(z^2f(z))=\sum_{n=0}^\infty(n+1)(n+2)\kk_n z^n
\end{equation*}
is in $L^1(\D)$. Moreover, we obtain upper and lower estimates for the $L^1(\D)$-norm of $g$ in terms of the trace class norm of $H$.

\begin{theorem}
	\label{int}
	Let $\D=\setw{z\in\C }{ |z|<1 }$ be the unit disc in the complex plane and assume that $(\kk_n)_{n=0}^\infty$ is a bounded sequence in $\C$. Consider the Hankel matrix $H=(h_{i,j})_{i,j\in\N_0}$ with entries $h_{i,j}=\kk_{i+j}$ for $i,j\in\N_0$ and the analytic function $g$ on $\D$ given by
	\begin{equation*}
		g(z)=\sum_{n=0}^\infty(n+2)(n+1)\kk_n z^n\qquad(z\in\D).
	\end{equation*}
	Then the following are equivalent:
	\begin{itemize}
		\item [(i)]$H$ is of trace class.
		\item [(ii)]$g\in L^1(\D)$.
	\end{itemize}
	If these two equivalent conditions are satisfied, then
	\begin{equation}
		\label{hic}
		h_{i,j}=\frac{1}{\pi}\int_\D g(\bar z)z^{i+j}(1-|z|^2)\dd z_1\dd z_2\qquad(i,j\in\N_0),
	\end{equation}
	where $z_1=\Re(z)$ and $z_2=\Im(z)$, and
	\begin{equation*}
		\|H\|_1\leq\|g\|_1\leq\frac{8}{\pi}\|H\|_1.
	\end{equation*}
\end{theorem}
\begin{proof}
	Assume first that $g\in L^1(\D)$. Using polar coordinates in $\D$ one finds that
	\begin{equation*}
		\frac{1}{\pi}\int_\D g(\bar z)z^{n}\dd z_1\dd z_2=\frac{(n+2)(n+1)}{n+1}\kk_n\qquad(n\in\N_0)
	\end{equation*}
	and
	\begin{equation*}
		\frac{1}{\pi}\int_\D g(\bar z)z^{n+1}\bar z\dd z_1\dd z_2=\frac{(n+2)(n+1)}{n+2}\kk_n\qquad(n\in\N_0),
	\end{equation*}
	from which~\eqref{hic} easily follows. Since the matrix whose $i,j$ th entries are given by $z^{i+j}$ for some $z\in\D$ has trace class norm given by $\frac{1}{1-|z|^2}$ (cf.~the proof of Theorem~\ref{sphSchurNorminf}) it follows from~\eqref{hic} that $H$ is of trace class and that
	\begin{equation*}
		\|H\|_1\leq\|g\|_1.
	\end{equation*}
	
	We now turn to the other implication of the theorem, and assume that $H$ is of trace class. By~\eqref{1*}--\eqref{1***} we can find sequences $(\xi^{(k)})_{k\in\N_0},(\eta^{(k)})_{k\in\N_0}$ in $\ell^2(\N_0)$ such that
	\begin{equation*}
		H=\sum_{k=0}^\infty\xi^{(k)}\odot\eta^{(k)}
	\end{equation*}
	and
	\begin{equation*}
		\|H\|_1=\sum_{k=0}^\infty\|\xi^{(k)}\|_2\|\eta^{(k)}\|_2,
	\end{equation*}
	from which it follows that
	\begin{equation}
		\label{Haga}
		\kk_n=\sum_{k=0}^\infty\xi_i^{(k)}\bar\eta_{n-i}^{(k)}\qquad(n\in\N_0),
	\end{equation}
	for any $i\in\{0,\cdots,n\}$.
	
	The following two Taylor expansions are easily verified:
	\begin{equation}
		\label{Taylor1}
		(1-z)^{-3}=\sum_{n=0}^\infty\frac{(n+1)(n+2)}{2}z^n\qquad(z\in\D)
	\end{equation}
	and
	\begin{equation}
		\label{Taylor2}
		(1-z)^{-3/2}=\sum_{n=0}^\infty\gamma_n z^n\qquad(z\in\D),
	\end{equation}
	where
	\begin{equation*}
		\gamma_n=\frac{\Gamma(n+\frac{3}{2})}{\Gamma(\frac{3}{2})\Gamma(n+1)}>0.
	\end{equation*}
	Using that the Gamma function is logarithmically convex on $\R^+$ (cf.~\cite[Chapter~5~(31)]{Ahl:ComplexAnalysis}) and therefore satisfy
	\begin{equation*}
		\Gamma(z+\tfrac{1}{2})\leq\Gamma(z)^{\frac{1}{2}}\Gamma(z+1)^{\frac{1}{2}}\qquad(z>0),
	\end{equation*}
	we have that
	\begin{equation*}
		\gamma_n\leq\frac{\Gamma(n+1)^{\frac{1}{2}}\Gamma(n+2)^{\frac{1}{2}}}{\Gamma(\frac{3}{2})\Gamma(n+1)}=\frac{2}{\sqrt\pi}\sqrt{n+1}\qquad(n\in\N_0).
	\end{equation*}
	Using $\big((1-z)^{-\tfrac{3}{2}}\big)^2=(1-z)^{-3}$ for $z\in\D$ together with~\eqref{Taylor1} and~\eqref{Taylor2} one finds that
	\begin{equation*}
		\sum_{i=0}^n\gamma_i\gamma_{n-i}=\frac{(n+1)(n+2)}{2}\qquad(n\in\N_0).
	\end{equation*}
	It is easy to check that the functions $(u_n)_{n\in\N_0}$ given by
	\begin{equation*}
		u_n(z)=z^n\qquad(z\in\D,\,n\in\N_0)
	\end{equation*}
	form an orthogonal sequence in $L^2(\D)$ with respect to the inner product
	\begin{equation*}
		\ip{\varphi}{\psi}=\frac{1}{\pi}\int_\D\varphi(z)\overline{\psi(z)}\dd z_1\dd z_2\qquad(\varphi,\psi\in L^2(\D)).
	\end{equation*}
	Moreover,
	\begin{equation*}
		\|u_n\|_2^2=\ip{u_n}{u_n}=\frac{1}{n+1}\qquad(n\in\N_0).
	\end{equation*}
	For $k\in\N_0$ put
	\begin{equation*}
		\varphi_k(z)=\sum_{n=0}^\infty\gamma_n\xi_n^{(k)}z^n\quad\mbox{and}\quad\psi_k(z)=\sum_{n=0}^\infty\gamma_n\bar\eta_n^{(k)}z^n\qquad(z\in\D).
	\end{equation*}
	Since
	\begin{equation*}
		\gamma_n^2\leq\frac{4}{\pi}(n+1)\qquad(n\in\N_0)
	\end{equation*}
	it follows that $\varphi_k,\psi_k\in L^2(\D)$. Moreover,
	\begin{equation*}
		\|\varphi_k\|_2^2\leq\frac{4}{\pi}\|\xi^{(k)}\|_2^2\quad\mbox{and}\quad\|\psi_k\|_2^2\leq\frac{4}{\pi}\|\eta^{(k)}\|_2^2
	\end{equation*}
	for all $k\in\N_0$. Hence, $\sum_{k=0}^\infty\varphi_k\psi_k\in L^1(\D)$ and
	\begin{equation*}
		\|\sum_{k=0}^\infty\varphi_k\psi_k\|_1\leq\frac{4}{\pi}\sum_{k=0}^\infty\|\xi^{(k)}\|_2\|\eta^{(k)}\|_2=\frac{4}{\pi}\|H\|_1.
	\end{equation*}
	For $z\in\D$,
	\begin{eqnarray*}
		\sum_{k=0}^\infty\varphi_k(z)\psi_k(z) & = & \sum_{i,j=0}^\infty\gamma_i\gamma_j\big(\sum_{k=0}^\infty\xi_i^{(k)}\bar\eta_j^{(k)}\big)z^{i+j}\\
		& = & \sum_{i,j=0}^\infty\gamma_i\gamma_j \kk_{i+j}z^{i+j}\\
		& = & \sum_{n=0}^\infty\big(\sum_{i=0}^\infty\gamma_i\gamma_{n-i}\big)\kk_n z^n\\
		& = & \frac{1}{2}\sum_{n=0}^\infty(n+1)(n+2)\kk_n z^n\\
		& = & \frac{1}{2}g(z).
	\end{eqnarray*}
	Hence $g\in L^1(\D)$ and
	\begin{equation*}
		\|g\|_1=2\|\sum_{k=0}^\infty\varphi_k\psi_k\|_1\leq\frac{8}{\pi}\|H\|_1.
	\end{equation*}
\end{proof}
\begin{theorem}
	\label{integralthm}
	Let $(\X,x_0)$ be a homogeneous tree of infinite degree with a distinguished vertex $x_0\in\X$. Let $\varphi:\X\to\C$ be a radial function and let $\dot\varphi:\N_0\to\C$ and $\tilde\varphi:\X\times\X\to\C$ be the corresponding functions as in Proposition~\ref{radial}. Then $\tilde\varphi$ is a Schur multiplier if and only if there exists constants $c_\pm\in\C$ and a complex Borel measure $\mu$ on $\D=\setw{z\in\C }{ |z|<1 }$ such that
	\begin{equation}
		\label{intform}
		\dot\varphi(n)=c_++c_-(-1)^n+\int_\D z^n\dd\mu(z)\qquad(n\in\N_0)
	\end{equation}
	and
	\begin{equation*}
		\int_\D \frac{|1-z^2|}{1-|z|^2}\dd|\mu|(z)<\infty.
	\end{equation*}
	Moreover,
	\begin{equation*}
		\|\tilde\varphi\|_S\leq|c_+|+|c_-|+\int_\D \frac{|1-z^2|}{1-|z|^2}\dd|\mu|(z)
	\end{equation*}
	and it is possible to choose $\mu$ such that
	\begin{equation}
		\label{2.19}
		|c_+|+|c_-|+\int_\D \frac{|1-z^2|}{1-|z|^2}\dd|\mu|(z)\leq\frac{8}{\pi}\|\tilde\varphi\|_S.
	\end{equation}
\end{theorem}
\begin{proof}
	If $\varphi$ has the form~\eqref{intform}, then the Hankel matrix $H=(h_{i,j})_{i,j\in\N_0}$ from Theorem~\ref{maintheoreminf} is given by
	\begin{equation*}
		h_{i,j}=\dot\varphi(i+j)-\dot\varphi(i+j+2)=\int_\D z^{i+j}(1-z^2)\dd\mu(z)\qquad(i,j\in\N_0),
	\end{equation*}
	from which it follows that
	\begin{equation*}
		\|H\|_1\leq\int_\D\frac{|1-z^2|}{1-|z|^2}\dd|\mu|(z),
	\end{equation*}
	where we again used that the matrix whose $i,j$ th entries are given by $z^{i+j}$ for some $z\in\D$ has trace class norm given by $\frac{1}{1-|z|^2}$. By assumption this is finite, so it follows from Theorem~\ref{maintheoreminf} that $\tilde\varphi$ is a Schur multiplier with
	\begin{equation*}
		\|\tilde\varphi\|_S\leq|c_+|+|c_-|+\int_\D \frac{|1-z^2|}{1-|z|^2}\dd|\mu|(z).
	\end{equation*}
	
	To prove the remaining part of the theorem, assume that $\tilde\varphi$ is a Schur multiplier and let $c_\pm,\psi$ and $H$ be defined as in Theorem~\ref{maintheoreminf}. Then
	\begin{equation*}
		\dot\varphi(n)=c_++c_-(-1)^n+\dot\psi(n)\qquad(n\in\N_0),
	\end{equation*}
	where
	\begin{equation*}
		\lim_{n\to\infty}\dot\psi(n)=0.
	\end{equation*}
	Moreover, $H=(h_{i,j})_{i,j\in\N_0}$ is a Hankel matrix of trace class with entries
	\begin{equation*}
		h_{i,j}=\dot\varphi(i+j)-\dot\varphi(i+j+2)=\dot\psi(i+j)-\dot\psi(i+j+2)\qquad(i,j\in\N_0)
	\end{equation*}
	and
	\begin{equation}
		\label{2.20}
		\|\tilde\varphi\|_S=|c_+|+|c_-|+\|H\|_1.
	\end{equation}
	By Theorem~\ref{int}, there exists a function $g\in L^1(\D)$ such that $\|g\|_1\leq\frac{8}{\pi}\|H\|_1$ and
	\begin{equation*}
		\dot\psi(n)-\dot\psi(n+2)=\frac{1}{\pi}\int_\D z^n(1-|z|^2)g(\bar z)\dd z_1\dd z_2\qquad(n\in\N_0).
	\end{equation*}
	Hence, for $n\in\N_0$ and $k\in\N$,
	\begin{eqnarray*}
		\dot\psi(n)-\dot\psi(n+2k) & = & \frac{1}{\pi}\sum_{j=0}^{k-1}\int_\D z^{n+2j}(1-|z|^2)g(\bar z)\dd z_1\dd z_2\\
		& = & \frac{1}{\pi}\int_\D\frac{z^n(1-z^{2k})}{1-z^2}(1-|z|^2)g(\bar z)\dd z_1\dd z_2.
	\end{eqnarray*}
	In the limit $k$ going to infinity we get, by Lebesgue's dominated convergence theorem, that
	\begin{equation*}
		\dot\psi(n)=\frac{1}{\pi}\int_\D z^n\frac{1-|z|^2}{1-z^2}g(\bar z)\dd z_1\dd z_2\qquad(n\in\N_0).
	\end{equation*}
	Therefore, \eqref{intform} holds with respect to the complex measure
	\begin{equation*}
		\dd\mu(z)=\frac{1}{\pi}\frac{1-|z|^2}{1-z^2}g(\bar z)\dd z_1\dd z_2.
	\end{equation*}
	Moreover,
	\begin{equation*}
		\int_\D\frac{|1-z^2|}{1-|z|^2}\dd|\mu|(z)=\|g\|_1\leq\frac{8}{\pi}\|H\|_1.
	\end{equation*}
	Hence, by~\eqref{2.20}
	\begin{equation*}
		|c_+|+|c_-|+\int_\D\frac{|1-z^2|}{1-|z|^2}\dd|\mu|(z)\leq\frac{8}{\pi}\|\tilde\varphi\|_S.
	\end{equation*}
\end{proof}
\begin{remark}
	\label{2.3.3}
	Theorem~\ref{integralthm} also holds for homogeneous tress $(\X,x_0)$ of finite degree $q+1$ ($2\leq q<\infty$) if one replaces the right hand side of~\eqref{2.19} by $\frac{8}{\pi}\frac{q+1}{q-1}\|\tilde\varphi\|_S$. This is an immediate consequence of Corollary~\ref{sameDifference}.
\end{remark}

	\section{Applications to free groups}
\label{fn}
Throughout this section $\Gamma$ denotes a group of the form
\begin{equation}
	\label{freeconvolution}
	\Gamma=\FF,
\end{equation}
where $M,N\in\N_0\bigcup\{\infty\}$ and $q=M+2N-1\geq 2$. In particular, this includes the groups
\begin{equation*}
	\ast_{m=1}^M\Z/2\Z\qquad(3\leq M\leq\infty)
\end{equation*}
and the (non-abelian) free groups
\begin{equation*}
	\F_N=\ast_{n=1}^N\Z\qquad(2\leq N\leq\infty).
\end{equation*}
By~\cite[page~16--18]{FTN:HarmonicAnalysisAndRepresentationTheoryForGroupsActingOnHhomogeneousTrees} the Cayley graph of $\Gamma$ is a homogeneous tree of degree $q+1$. There is a canonical distinguished vertex $x_0$ in $\Gamma$, namely the identity element $e$. The results of section~\ref{homtree} and~\ref{sphfct} can all be reformulated as results about radial or spherical functions on $(\Gamma,e)$, but the concept of a Schur multiplier is perhaps more naturally replaced by the concept of a completely bounded Fourier multiplier (written $\MoA(\Gamma)$).
\begin{proposition}
	\label{equalnormsF}
	Consider a group $\Gamma$ of the form~\eqref{freeconvolution} with $2\leq q\leq\infty$. Let $\varphi$ be a radial function on $\Gamma$, then $\varphi$ is a completely bounded Fourier multiplier of $\Gamma$ if and only if the corresponding function $\tilde\varphi:\Gamma\times\Gamma\to\C$ given by Proposition~\ref{radial} is a Schur multiplier. Moreover,
	\begin{equation*}
		\|\varphi\|_{\MoA(\Gamma)}=\|\tilde\varphi\|_S.
	\end{equation*}
\end{proposition}
\begin{proof}
	By left invariance of the metric $\dd$ on $\Gamma$, we have
	\begin{equation*}
		\tilde\varphi(x,y)=\dot\varphi(\dd(x,y))=\dot\varphi(\dd(y^{-1}x,e))=\varphi(y^{-1}x)=\hat\varphi(x,y)\qquad(x,y\in\Gamma),
	\end{equation*}
	where $\hat\varphi:\Gamma\times\Gamma\to\C$ is given by~\eqref{new0.3}. Hence, Proposition~\ref{equalnormsF} follows from Proposition~\ref{Gilbert0} and the equalities
	\begin{equation*}
		\|\varphi\|_{\MoA(\Gamma)}=\|\varphi\|_{HS}=\|\hat\varphi\|_{S}.
	\end{equation*}
\end{proof}
Since the spherical functions on $\Gamma$ are simply the spherical functions on the homogeneous tree $(\Gamma,e)$, where we have identified (the vertices of) the Cayley graph with $\Gamma$, we can use Proposition~\ref{equalnormsF} to reformulate the main results from section~\ref{homtree}, \ref{sphfct} and~\ref{integral} (i.e., Theorem~\ref{maintheorem}, \ref{maintheoreminf}, \ref{sphSchurNormEigenvalue}, \ref{sphSchurNorminf} and \ref{integralthm} and Remark~\ref{2.3.3}).
\begin{theorem}
	\label{maintheoremgamma}
	Consider a group $\Gamma$ of the form~\eqref{freeconvolution} with $2\leq q\leq\infty$. Let $\varphi:\Gamma\to\C$ be a radial function and let $\dot\varphi:\N_0\to\C$ be the corresponding function as in Proposition~\ref{radial}. Finally, let $H=(h_{i,j})_{i,j\in\N_0}$ be the Hankel matrix given by $h_{i,j}=\dot\varphi(i+j)-\dot\varphi(i+j+2)$ for $i,j\in\N_0$. Then the following are equivalent:
	\begin{itemize}
		\item [(i)]$\varphi$ is a completely bounded Fourier multiplier of $\Gamma$.
		\item [(ii)]$H$ is of trace class.
	\end{itemize}
	If these two equivalent conditions are satisfied, then there exists unique constants $c_\pm\in\C$ and a unique $\dot\psi:\N_0\to\C$ such that
	\begin{equation*}
		\dot\varphi(n)=c_++c_-(-1)^n+\dot\psi(n)\qquad(n\in\N_0)
	\end{equation*}
	and
	\begin{equation*}
		\lim_{n\to\infty}\dot\psi(n)=0.
	\end{equation*}
	Moreover,
	\begin{equation*}
		\|\varphi\|_{\MoA(\Gamma)}=|c_+|+|c_-|+\left\{
		\begin{array}{lll}
			\|H\|_1 & \mbox{if} & q=\infty\\
			\big(1-\tfrac{1}{q}\big)\|\big(I-\tfrac{\tau}{q}\big)^{-1}H\|_1 & \mbox{if} & q<\infty,
		\end{array}
		\right.
	\end{equation*}
	where $\tau$ is the shift operator defined by~\eqref{tau}.
\end{theorem}
\begin{theorem}
	Consider a group $\Gamma$ of the form~\eqref{freeconvolution} with $2\leq q\leq\infty$. Let $\varphi:\Gamma\to\C$ be a radial function and let $\dot\varphi:\N_0\to\C$ be the corresponding function as in Proposition~\ref{radial}. Then $\varphi$ is a completely bounded Fourier multiplier of $\Gamma$ if and only if there exists constants $c_\pm\in\C$ and a complex Borel measure $\mu$ on $\D=\setw{z\in\C }{ |z|<1 }$ such that
	\begin{equation*}
		\dot\varphi(n)=c_++c_-(-1)^n+\int_\D z^n\dd\mu(z)\qquad(n\in\N_0)
	\end{equation*}
	and
	\begin{equation*}
		\int_\D \frac{|1-z^2|}{1-|z|^2}\dd|\mu|(z)<\infty.
	\end{equation*}
	Moreover,
	\begin{equation*}
		\|\varphi\|_{\MoA(\Gamma)}\leq|c_+|+|c_-|+\int_\D \frac{|1-z^2|}{1-|z|^2}\dd|\mu|(z)
	\end{equation*}
	and it is possible to choose $\mu$ such that
	\begin{equation*}
		|c_+|+|c_-|+\int_\D \frac{|1-z^2|}{1-|z|^2}\dd|\mu|(z)\leq\frac{8}{\pi}\frac{q+1}{q-1}\|\varphi\|_{\MoA(\Gamma)},
	\end{equation*}
	where we set $\frac{q+1}{q-1}$ equal to $1$ when $q=\infty$.
\end{theorem}
\begin{theorem}
	\label{fnthm}
	Consider a group $\Gamma$ of the form~\eqref{freeconvolution} with $2\leq q\leq\infty$. Let $\varphi:\Gamma\to\C$ be a spherical function, then $\varphi$ is a completely bounded Fourier multiplier of $\Gamma$ if and only if the eigenvalue $\eval$ corresponding to $\varphi$ is in the set
	\begin{equation*}
		\setw{\eval\in\C }{ \Re(\eval)^2+\big(\tfrac{q+1}{q-1}\big)^2\Im(\eval)^2<1 }\bigcup\seto{\pm1}.
	\end{equation*}
	The corresponding norm is given by
	\begin{equation*}
		\|\varphi\|_{\MoA(\Gamma)}=\frac{|1-\eval^2|}{1-\Re(\eval)^2-\left(\frac{q+1}{q-1}\right)^2\Im(\eval)^2}\qquad(\Re(\eval)^2+\big(\tfrac{q+1}{q-1}\big)^2\Im(\eval)^2<1)
	\end{equation*}
	and
	\begin{equation*}
		\|\varphi\|_{\MoA(\Gamma)}=1\qquad(\eval=\pm1),
	\end{equation*}
	where we set $\frac{q+1}{q-1}$ equal to $1$ when $q=\infty$.
\end{theorem}
\begin{remark}
	The case $q=\infty$ and $M=0$ of Theorem~\ref{fnthm} was proved by Pytlik and Szwarc in~\cite[Corollary~4]{PS:AnAnalyticFamilyOfUniformlyBoundedRepresentationsOfFreegroups}.
\end{remark}
\begin{corollary}
	\label{MainThmCorollary}
	Consider a group $\Gamma$ of the form~\eqref{freeconvolution} with $2\leq q\leq\infty$. There is no uniform bound on the ${\MoA(\Gamma)}$-norm of the spherical functions on $(\Gamma,e)$ which are completely bounded Fourier multipliers.
\end{corollary}
\begin{lemma}
	\label{MAG}
	Let $\varphi$ be a radial function on $\Gamma$. If
	\begin{equation*}
		\sum_{n=0}^\infty(n+1)^2|\dot\varphi(n)|^2<\infty,
	\end{equation*}
	then $\varphi\in\MA(\Gamma)$. Moreover,
	\begin{equation*}
		\|\varphi\|_{\MA(\Gamma)}\leq\big(\sum_{n=0}^\infty(n+1)^2|\dot\varphi(n)|^2\big)^{\frac{1}{2}}.
	\end{equation*}
\end{lemma}
\begin{proof}
	According to~\cite[Proposition~1.2]{DCH:MultipliersOfTheFourierAlgebrasOfSomeSimpleLieGroupsAndTheirDiscreteSubgroups} we have to show that $\varphi$ is bounded (which is obvious from the assumption in the lemma) and that
	\begin{equation*}
		\|\lambda(\varphi f)\|\leq\big(\sum_{n=0}^\infty(n+1)^2|\dot\varphi(n)|^2\big)^{\frac{1}{2}}\|\lambda(f)\|\qquad(f\in\ell^1(\Gamma)),
	\end{equation*}
	where $\lambda:\ell^2(\Gamma)\to\linbeg(\ell^2(\Gamma))$ is the left regular representation. Following~\cite{Haa:AnExampleOfANonnuclearC*-algebraWhichHasTheMetricApproximationProperty} we let $\unit_n$ denote the characteristic function of the set $\setw{x\in\Gamma }{ \dd(x,x_0)=n }$ for $n\in\N_0$. If $\Gamma$ is the free group $\F_N$ on $N$ generators ($2\leq N<\infty$), then by~\cite[Lemma~1.4]{Haa:AnExampleOfANonnuclearC*-algebraWhichHasTheMetricApproximationProperty}
	\begin{equation}
		\label{3.1.5}
		\|\lambda(\varphi f)\|\leq\sum_{n=0}^\infty(n+1)|\dot\varphi(n)|\|f\unit_n\|_2\qquad(f\in\cont_\cpt(\Gamma)).
	\end{equation}
	Using~\cite[Theorem~5.1]{BP:HarmonicAnalysisForGroupsActingOnTrees}, the same inequality holds for $\Gamma$ of the form~\eqref{freeconvolution} when $q<\infty$, and by a simple inductive limit argument, \eqref{3.1.5} also holds when $q=\infty$. By the Cauchy--Schwarz, inequality~\eqref{3.1.5} implies that
	\begin{eqnarray}
		\label{3.1.75}
		\|\lambda(\varphi f)\| & \leq & \big(\sum_{n=0}^\infty(n+1)^2|\dot\varphi(n)|^2\big)^{\frac{1}{2}}\big(\sum_{n=0}^\infty\|f\unit_n\|_2^2 \big)^{\frac{1}{2}}\\
		\nonumber & = & \big(\sum_{n=0}^\infty(n+1)^2|\dot\varphi(n)|^2\big)^{\frac{1}{2}}\|f\|_2\\
		\nonumber & \leq & \big(\sum_{n=0}^\infty(n+1)^2|\dot\varphi(n)|^2\big)^{\frac{1}{2}}\|\lambda(f)\|
	\end{eqnarray}
	for $f\in\cont_\cpt(\Gamma)$, because $\|f\|_2=\|\lambda(f)\delta_{x_0}\|_2\leq\|\lambda(f)\|$. Since $\cont_\cpt(\Gamma)$ is dense in $\ell^1(\Gamma)$, \eqref{3.1.75} holds for all $f\in\ell^1(\Gamma)$. This finishes the prof of the lemma.
\end{proof}
\begin{proposition}
	\label{existsradial}
	Consider a group $\Gamma$ of the form~\eqref{freeconvolution} with $2\leq q\leq\infty$. There exists a radial function $\varphi:\Gamma\to\C$ such that $\varphi$ is a Fourier multiplier of $\Gamma$, but not a completely bounded Fourier multiplier, i.e., $\varphi\in\MA(\Gamma)\setminus\MoA(\Gamma)$.
\end{proposition}
\begin{proof}
	Let $\varphi:\Gamma\to\C$ be the radial function given by
	\begin{equation*}
		\dot\varphi(n)=\left\{
		\begin{array}{lll}
			\alpha_k & \mbox{if} & n=2^k\mbox{ for some }k\in\N\\
			0 & \mbox{if} & n=2^k\mbox{ for all }k\in\N,
		\end{array}
		\right.
	\end{equation*}
	where
	\begin{equation*}
		\alpha_k=\frac{1}{k\cdot2^k}\qquad(k\in\N).
	\end{equation*}
	To show that $\varphi\in\MA(\Gamma)$ use Lemma~\ref{MAG} and verify that
	\begin{equation*}
		\sum_{n=0}^\infty(n+1)^2|\dot\varphi(n)|^2=\sum_{k=1}^\infty\Big(\frac{2^k+1}{k\cdot2^k}\Big)^2\leq\frac{9}{4}\sum_{k=1}^\infty\frac{1}{k^2}=\frac{3}{8}\pi^2<\infty.
	\end{equation*}
	In order to see that $\varphi$ is not a completely bounded Fourier multiplier, we have to show that the Hankel matrix $H$ of Theorem~\ref{maintheoremgamma} is not of trace class.
	
	Assume that $H$ is of trace class. Let $\setw{e_i }{ i\in\N_0 }$ be the standard basis of $\ell^2(\N_0)$ and put
	\begin{equation*}
		E_k=\sspan\setw{e_i }{ 3\cdot2^{k-3}\leq i\leq5\cdot2^{k-3}}\qquad(k\geq3).
	\end{equation*}
	Note that $(E_k)_{k=3}^\infty$ is a sequence of mutually orthogonal subspaces of $\ell^2(\N_0)$. Let $P_k$ denote the orthogonal projection of $\ell^2(\N_0)$ onto $E_k$. Then
	\begin{equation*}
		\|H\|_1\geq\|\sum_{k=3}^\infty P_kHP_k\|_1=\sum_{k=3}^\infty \|P_kHP_k\|_1.
	\end{equation*}
	However, $\|P_kHP_k\|_1$ is the trace class norm of the $(2^{k-2}+1)\times(2^{k-2}+1)$ sub-matrix of $H$ corresponding to row and column indices $i,j$ satisfying
	\begin{equation*}
		3\cdot2^{k-3}\leq i,j\leq5\cdot2^{k-3}\qquad(i,j\in\N_0).
	\end{equation*}
	Note that all the entries $h_{i,j}$ of the non-main (or anti-) diagonal of this sub-matrix are equal to
	\begin{equation*}
		\dot\varphi(2^k)-\dot\varphi(2^k+2)=\alpha_k-0=\frac{1}{k\cdot2^k}\qquad(k\geq3).
	\end{equation*}
	Hence,
	\begin{equation*}
		\|P_kHP_k\|_1\geq\frac{2^{k-2}+1}{k\cdot2^k}\geq\frac{1}{4k}\qquad(k\geq3)
	\end{equation*}
	and therefore
	\begin{equation*}
		\sum_{k=3}^\infty\|P_kHP_k\|_1=\infty,
	\end{equation*}
	which contradicts the fact that $H$ is of trace class. Therefore $\varphi\notin\MoA(\Gamma)$.
\end{proof}

	\section{\texorpdfstring{Applications to $\PGLQ$}{Applications to PGL2(Qq)}}
\label{pgl2qq}
Let $q$ be some prime number and let $|\cdot|_q:\Q\to\R^+_0$ be the \emph{p-adic norm} (corresponding to $q$) given by
\begin{equation*}
	|0|_q=0\quad\mbox{and}\quad |q^n\frac{s}{t}|_q=q^{-n}\qquad(n\in\Z),
\end{equation*}
when $s,t\in\Z$ are not divisible by $q$. The following relations are well known:
\begin{itemize}
	\item [(i)]$|x y|_q=|x|_q|y|_q\qquad(x,y\in\Q)$.
	\item [(ii)]$|x+y|_q\leq\max\seto{|x|_q,|y|_q}\qquad(x,y\in\Q)$.
\end{itemize}
Property {\rm (ii)} is referred to as the \emph{ultrametric inequality} since it implies the triangle inequality. The \emph{p-adic metric} $\dd_q:\Q\times\Q\to\R^+_0$ (corresponding to $q$) is defined by
\begin{equation*}
	\dd_q(x,y)=|x-y|_q\qquad(x,y\in\Q).
\end{equation*}
The completion of $\Q$ in this metric is written $\padicnum$ and referred to as the \emph{p-adic numbers} (corresponding to $q$). The p-adic norm and the p-adic metric have natural extensions to $\padicnum$ and the properties {\rm (i)} and {\rm (ii)} hold for all $x,y\in\padicnum$.

We now list some standard properties of the p-adic numbers and subsets thereof (we refer to \cite[Appendix~\S1 and \S2]{FTN:HarmonicAnalysisAndRepresentationTheoryForGroupsActingOnHhomogeneousTrees} for the proofs). Let $\padicnums$ denote the group of invertible elements in $\padicnum$, that is, the non-zero p-adic numbers. Each $a\in\padicnums$ can be written uniquely as the (formal) sum
\begin{equation*}
	a=\sum_{i=k}^\infty a_i q^i\qquad(k\in\Z,\,a_i\in\{0,1,\ldots,q-1\},\,a_k\neq0),
\end{equation*}
where we note that $|a|_q=q^{-k}$.

By $\padicint$ we denote the subring of $\padicnum$ consisting of \emph{p-adic integers} (corresponding to $q$), that is, elements $a\in\padicnum$ with $|a|_q\leq1$. Let $\padicints$ denote the invertible elements in $\padicint$, i.e., $a\in\padicints$ if an only if $a\in\padicint\setminus\{0\}$ and $a^{-1}\in\padicint$. Hence, $\padicints$ is the set of p-adic numbers $a\in\padicnum$ for which $|a|_q=1$. These elements are referred to as \emph{p-adic units} (corresponding to $q$), and they obviously form a subgroup of $\padicnum$. We note that if $a$ is a p-adic unit and $n\in\Z$ then $|q^n a|_q=q^{-n}$. Therefore, $\padicnums$ is the disjoint union
\begin{equation}
	\label{QZ}
	\padicnums=\bigsqcup_{n=-\infty}^\infty q^n\padicints.
\end{equation}

Denote by $\GLQ$ the set of $2\times2$ matrices with entries from $\padicnum$ and non-zero determinant, and denote by $\SLZ$ the set of $2\times2$ matrices with entries from $\padicint$ and unit determinant. Given $A\in\GLQ$ it is a fact, which will be used frequently without further mentioning, that $A\padicint^2=\padicint^2$ if and only if $A\in\SLZ$, where $\padicint^2$ is shorthand notation for $\padicint\oplus\padicint$. Let $\PGLQ$ denote the quotient of $\GLQ$ by its center $\padicnums I$, where $I$ denotes the ($2\times2$) identity matrix. Similarly, we let $\PSLZ$ denote the quotient of $\SLZ$ by its center $\padicints I$. Let $\pi:\GLQ\to\PGLQ$ be the quotient map given by
\begin{equation*}
	\pi(A)=\padicnums A\qquad(A\in\GLQ).
\end{equation*}
We claim that the map
\begin{equation*}
	\pi(V)\mapsto\padicints V
\end{equation*}
is a well defined bijection from $\pi(\SLZ)$ to $\PSLZ$, thereby showing that these two sets are isomorphic. Henceforth, we consider $\PSLZ$ as a subset of $\PGLQ$ and we note that it is both compact and open (cf.~\cite{Mau:SphericalFunctionsOverP-adicFields.I}). The only non-trivial part is to show that the map is well defined. To this end, assume that $V,W\in\SLZ$ with $\pi(V)=\pi(W)$, which implies the existence of some $a\in\padicnums$ such that $W=aV$ and therefore that $\det(W)=a^2\det(V)$. But $V$ and $W$ both have unit determinant, so we conclude that $a^2$ is a p-adic unit from which it follows that $a$ is also a p-adic unit. This finishes the argument.

A \emph{lattice} (of $\padicnum^2$) is a set of the form
\begin{equation*}
	\padicint\vec e_1+\padicint\vec e_2,
\end{equation*}
where $\vec e_1,\vec e_2\in\padicnum^2$ form a basis for $\padicnum^2$. This set can also be written $(\vec e_1\ \vec e_2)\padicint^2$, where $(\vec e_1\ \vec e_2)$ denotes the matrix with column vectors $\vec e_1$ and $\vec e_2$. Two lattices $L,L'$ are called \emph{equivalent} (written $L\sim L'$) if there exists $a\in\padicnums$ such that $L'=aL$. Since obviously $aL=L$ for any lattice $L$ when $a\in\padicints$, one concludes from~\eqref{QZ} that two lattices $L,L'$ are equivalent if and only if there exists $n\in\Z$ such that $L'=q^nL$. Denote the set of lattices by $\cL$, and denote the set of equivalence classes of lattices by $\cL/\sim$.
\begin{lemma}
	\label{3X}
	There are natural bijective maps between the following three sets\footnote{The sets in {\it (i)} and {\it (ii)} are the sets of left cossets.}
	\begin{itemize}
		\item [(i)]$\PGLQ/\PSLZ$
		\item [(ii)]$\GLQ/\bigsqcup_{n=-\infty}^\infty q^n\SLZ$
		\item [(iii)]$\cL/\sim$.
	\end{itemize}
	More specifically, the following two maps give rise to bijections from {\it (ii)} to {\it (i)} and {\it (ii)} to {\it (iii)}, respectively:
	\begin{equation*}
		A\mapsto[\padicnums A]\quad\mbox{and}\quad A\mapsto[A\padicint^2]\qquad(A\in\GLQ),
	\end{equation*}
	where the brackets denote the corresponding equivalence classes in the quotient $\PGLQ/\PSLZ$ and $\cL/\sim$, respectively.
\end{lemma}
\begin{proof}
	This is elementary, and the details of the proof will be left to the reader. For the bijection between {\it (i)} and {\it (ii)}, one just have to check that the kernel of the composition of the two quotient maps:
	\begin{equation*}
		\GLQ\to\PGLQ\to\PGLQ/\PSLZ
	\end{equation*}
	is equal to $\bigsqcup_{n=-\infty}^\infty q^n\SLZ$. And for the bijection between {\it (ii)} and {\it (iii)}, one shows first that the map $A\mapsto[\padicnums A]$ gives rise to a bijection of the quotient $\GLQ/\SLZ$ onto $\cL$.
\end{proof}
Let $\X$ denote the set from Lemma~\ref{3X} and notice that $\X$ is discrete since $\PSLZ$ is open. The characterization {\it (i)} is useful since spherical functions on the Gelfand pair $(\PGLQ,\PSLZ)$ have been studied elsewhere (cf.~\cite{Mau:SphericalFunctionsOverP-adicFields.I}). The characterization {\it (iii)} is used for introducing the tree structure to $\X$ (cf.~\cite[Appendix~\S4 and \S5]{FTN:HarmonicAnalysisAndRepresentationTheoryForGroupsActingOnHhomogeneousTrees} and \cite[Chapter~II~\S1]{Ser:ArbresAmalgamesSL2}). Finally, the characterization {\it (ii)} is useful for doing actual calculations. We denote the elements of $\X$ by $\Lambda$. Unless we explicitly specify an element of $\X$, by writing up its equivalence class, we let our choice of representative reveal which of the three pictures we are working in. For instance, we let $\Lambda_0$ be the element in $\X$ which has (canonical) representatives $\padicnums I\in\PGLQ$, $I\in\GLQ$ and $I\padicint^2\in\cL$.

In the following we use the notation $\G=\PGLQ$ and $\K=\PSLZ$ ($\Lambda_0$ is $\K$ in the characterization {\it (i)}). Obviously, $\G$ induces a left action on $\X$ which is compatible with the different characterizations from Lemma~\ref{3X}, in fact, if $\Lambda\in X$ is represented by $A\in\GLQ$ and $g\in\G$ is represented by $B\in\GLQ$, then $BA\in\GLQ$ represents $g\Lambda$. It is well known (cf.~\cite[Chapter~II~\S1]{Ser:ArbresAmalgamesSL2}) that $\X$ can be interpreted as a homogeneous tree of degree $q+1$ by introducing a certain metric $\dd$ on $\X$, which will be described below.

Let $\Lambda,\Lambda'\in X$ with representatives $L,L'\in\cL$. According to \cite[Appendix Theorem~4.3 and~\S5]{FTN:HarmonicAnalysisAndRepresentationTheoryForGroupsActingOnHhomogeneousTrees} there exists linearly independent vectors $\vec e_1,\vec e_2\in\padicnum^2$ and $i,j\in\Z$ such that
\begin{equation}
	\label{5.2}
	\padicint\vec e_1+\padicint\vec e_2=L\quad\mbox{and}\quad q^i\padicint\vec e_1+q^j\padicint\vec e_2=L'.
\end{equation}
Moreover, the number $|i-j|$ only depend on $\Lambda$ and $\Lambda'$. The metric $\dd$ on $\X$, which turns $\X$ into a homogeneous tree of degree $q+1$, is given by
\begin{equation}
	\label{5.25}
	\dd(\Lambda,\Lambda')=|i-j|.
\end{equation}
In particular, $\Lambda,\Lambda'\in\X$ are connected by an edge if an only if $\dd(\Lambda,\Lambda')=1$. Note that in~\eqref{5.2} one can always assume that $i\geq j$ (by replacing $(\vec e_1,\vec e_2)$ with $(\vec e_2,\vec e_1)$ if $i<j$). In this case, \eqref{5.2} can be rewritten as
\begin{equation}
	\label{5.3}
	\padicint\vec e_1+\padicint\vec e_2=L\quad\mbox{and}\quad q^n\padicint\vec e_1+\padicint\vec e_2=q^mL',
\end{equation}
where $n=i-j=\dd(\Lambda,\Lambda')$ and $m=-j$. From this we get:
\begin{lemma}
	Let $\Lambda,\Lambda'\in X$ with representatives $A,A'\in\GLQ$. Then there exists linearly independent vectors $\vec e_1,\vec e_2\in\padicnum^2$ and $m\in\Z$ such that
	\begin{equation}
		\label{5.4}
		\padicint\vec e_1+\padicint\vec e_2=A\padicint^2\quad\mbox{and}\quad q^n\padicint\vec e_1+\padicint\vec e_2=q^mA'\padicint^2,
	\end{equation}
	where $n=\dd(\Lambda,\Lambda')$.
\end{lemma}
\begin{proof}
	Let $\vec e_1,\vec e_2,i,j$ be as in~\eqref{5.2} with $i\geq j$, and put $n=i-j=\dd(\Lambda,\Lambda')$ and $m=-j$. Since $A\padicint^2,A'\padicint^2$ are representatives of $\Lambda,\Lambda'$ in $\cL$, \eqref{5.4} follows immediately from~\eqref{5.3}.
\end{proof}
The action of $\G$ on $\X$ is an isometry, so we have that
\begin{equation*}
	\dd(\Lambda,\Lambda_0)=\dd(k\Lambda,\Lambda_0)\qquad(\Lambda\in\X,\,k\in\K)
\end{equation*}
because $k\Lambda_0=\Lambda_0$ for all $k\in\K$. We also have the converse.
\begin{lemma}
	\label{eqdist}
	If $\Lambda,\Lambda'\in X$ satisfy
	\begin{equation*}
		\dd(\Lambda,\Lambda_0)=\dd(\Lambda',\Lambda_0),
	\end{equation*}
	then there exists $k\in\K$ satisfying $k\Lambda=\Lambda'$.
\end{lemma}
\begin{proof}
	Put $n=\dd(\Lambda,\Lambda_0)=\dd(\Lambda',\Lambda_0)$, let $A,A'\in\GLQ$ be representatives for $\Lambda,\Lambda'$ and use $I\in\GLQ$ as a representative for $\Lambda_0$. Applying~\eqref{5.4} to $\Lambda_0,\Lambda$ and $\Lambda_0,\Lambda'$, respectively, we find vectors $\vec e_1,\vec e_2,\vec f_1,\vec f_2\in\padicnum^2$ such that
	\begin{equation*}
		\padicint\vec e_1+\padicint\vec e_2=I\padicint^2=\padicint^2\quad\mbox{and}\quad q^n\padicint\vec e_1+\padicint\vec e_2=q^mA\padicint^2
	\end{equation*}
	\begin{equation*}
		\padicint\vec f_1+\padicint\vec f_2=I\padicint^2=\padicint^2\quad\mbox{and}\quad q^n\padicint\vec f_1+\padicint\vec f_2=q^{m'}A'\padicint^2,
	\end{equation*}
	for some $m,m'\in\Z$. Let $V$ denote the matrix representing the change of basis sending $\vec e_i$ to $\vec f_i$ for $i=1,2$. From the above expressions we conclude that $V\padicint^2=\padicint^2$ and therefore $V\in\SLZ$. Observe that
	\begin{equation*}
		q^{m-m'}VA\padicint^2=A'\padicint^2,
	\end{equation*}
	and conclude that $k \Lambda=\Lambda'$ when
	\begin{equation*}
		k=\padicnums V\in\K.
	\end{equation*}
\end{proof}
A function $f$ on $\G$ is called \emph{$\K$-bi-invariant} if
\begin{equation*}
	f(k g k')=f(g)\qquad(g\in\G,\,k,k'\in\K).
\end{equation*}
Using the above lemma and that $\K$ is an open subgroup of $\G$, we conclude that there is a bijective correspondence between continuous $\K$-bi-invariant functions $\varphi_{\G}$ on $\G$ and radial functions $\varphi_{\X}$ on $\X$ given by
\begin{equation}
	\label{phipsi}
	\varphi_{\G}(g)=\varphi_{\X}(g\Lambda_0)\qquad(g\in\G).
\end{equation}
\begin{lemma}
	\label{BF*}
	If $\varphi_{\G}$ and $\varphi_\X$ are related as in~\eqref{phipsi}, then
	\begin{equation*}
		\varphi_{\G}(g^{-1}g')=\tilde\varphi_{\X}(g'\Lambda_0,g\Lambda_0)\qquad(g,g'\in\G).
	\end{equation*}
\end{lemma}
\begin{proof}
	For $g,g'\in\G$ we find that
	\begin{eqnarray*}
		\varphi_{\G}(g^{-1}g') & = & \varphi_{\X}(g^{-1}g'\Lambda_0)\\
		& = & \dot\varphi_{\X}(\dd(g^{-1}g'\Lambda_0,\Lambda_0))\\
		& = & \dot\varphi_{\X}(\dd(g'\Lambda_0,g\Lambda_0))\\
		& = & \tilde\varphi_{\X}(g'\Lambda_0,g\Lambda_0).
	\end{eqnarray*}
\end{proof}
\begin{proposition}
	\label{equalnorms}
	Let $\varphi_{\G}$ be a continuous $\K$-bi-invariant function on $\G$, then $\varphi_{\G}$ is a completely bounded Fourier multiplier of $\G$ if and only if $\tilde\varphi_{\X}$ is a Schur multiplier. Moreover,
	\begin{equation*}
		\|\varphi_{\G}\|_{\MoA(\G)}=\|\tilde\varphi_{\X}\|_S.
	\end{equation*}
\end{proposition}
\begin{proof}
	Assume that $\tilde\varphi_{\X}$ is a Schur multiplier and use Proposition~\ref{Grothendieck} to find a Hilbert space $\Hil$ and bounded maps $\PX,\QX:\X\to\Hil$ such that
	\begin{equation*}
		\tilde\varphi_{\X}(x',x)=\ip{\PX(x')}{\QX(x)}\qquad(x,x'\in\X)
	\end{equation*}
	and
	\begin{equation*}
		\|\PX\|_\infty\|\QX\|_\infty=\|\tilde\varphi_{\X}\|_S.
	\end{equation*}
	Define bounded maps $\PG,\QG:\G\to\Hil$ by
	\begin{equation}
		\label{xieta}
		\PG(g')=\PX(g'\Lambda_0)\quad\mbox{and}\quad\QG(g)=\QX(g\Lambda_0)\qquad(g,g'\in\G)
	\end{equation}
	and use Lemma~\ref{BF*} to show that
	\begin{equation}
		\label{xietacalc}
		\varphi_{\G}(g^{-1}g')=\tilde\varphi_\X(g'\Lambda_0,g\Lambda_0)=\ip{\PX(g'\Lambda_0)}{\QX(g\Lambda_0)}=\ip{\PG(g')}{\QG(g)}
	\end{equation}
	for all $g,g'\in\G$. Using Proposition~\ref{Gilbert0} we conclude that $\varphi_{\G}$ is a completely bounded Fourier multiplier of $\G$, with
	\begin{equation*}
		\|\varphi_{\G}\|_{\MoA(\G)}\leq\|\PG\|_\infty\|\QG\|_\infty=\|\PX\|_\infty\|\QX\|_\infty=\|\tilde\varphi_{\X}\|_S.
	\end{equation*}
	
	Now assume that $\varphi_{\G}$ is a completely bounded Fourier multiplier of $\G$ and use Proposition~\ref{Gilbert0} to find a Hilbert space $\Hil$ and bounded maps $\PG,\QG:\G\to\Hil$ such that
	\begin{equation*}
		\varphi_{\G}(g^{-1}g')=\ip{\PG(g')}{\QG(g)}\qquad(g,g'\in G)
	\end{equation*}
	and
	\begin{equation*}
		\|\PG\|_\infty\|\QG\|_\infty=\|\varphi_{\G}\|_{\MoA(\G)}.
	\end{equation*}
	Let $\psi:\X\to\G$ be a \emph{cross section} of the map $g\mapsto g\Lambda_0$ of $\G$ onto $\X$, i.e., $\psi$ satisfies
	\begin{equation*}
		\psi(x)\Lambda_0=x\qquad(x\in\X).
	\end{equation*}
	Define bounded maps $\PX,\QX:\X\to\Hil$ by
	\begin{equation*}
		\PX(x')=\PG(\psi(x'))\quad\mbox{and}\quad \QX(x)=\QG(\psi(x))\qquad(x,x'\in\X)
	\end{equation*}
	and use Lemma~\ref{BF*} to show that
	\begin{equation*}
		\tilde\varphi_\X(x',x)=\tilde\varphi_\X(\psi(x')\Lambda_0,\psi(x)\Lambda_0)=\varphi_\G(\psi(x)^{-1}\psi(x'))=\ip{\PX(x')}{\QX(x)}
	\end{equation*}
	for all $x,x'\in\X$. Using Proposition~\ref{Grothendieck} we conclude that $\tilde\varphi_{\X}$ is a Schur multiplier, with
	\begin{equation*}
		\|\tilde\varphi_{\X}\|_S\leq\|\PX\|_\infty\|\QX\|_\infty\leq\|\PG\|_\infty\|\QG\|_\infty=\|\varphi_{\G}\|_{\MoA(\G)}.
	\end{equation*}
\end{proof}
Using Proposition~\ref{equalnorms} we obtain the following from Theorem~\ref{maintheorem}.
\begin{maintheorem}
	\label{2.5.6}
	Let $q$ be a prime number and consider the groups $\G=\PGLQ$ and $\K=\PSLZ$ and their quotient $\X=\G/\K$. Let $\varphi_{\G}:\G\to\C$ be a continuous $\K$-bi-invariant function and let $\varphi_{\X}:\X\to\C$ be the corresponding function as in~\eqref{phipsi} and $\dot\varphi_{\X}:\N_0\to\C$ the corresponding function as in Proposition~\ref{radial}. Finally, let $H=(h_{i,j})_{i,j\in\N_0}$ be the Hankel matrix given by $h_{i,j}=\dot\varphi(i+j)-\dot\varphi(i+j+2)$ for $i,j\in\N_0$. Then the following are equivalent:
	\begin{itemize}
		\item [(i)]$\varphi_{\G}$ is a completely bounded Fourier multiplier of $\G$.
		\item [(ii)]$H$ is of trace class.
	\end{itemize}
	If these two equivalent conditions are satisfied, then there exists unique constants $c_\pm\in\C$ and a unique $\dot\psi:\N_0\to\C$ such that
	\begin{equation*}
		\dot\varphi_{\X}(n)=c_++c_-(-1)^n+\dot\psi(n)\qquad(n\in\N_0)
	\end{equation*}
	and
	\begin{equation*}
		\lim_{n\to\infty}\dot\psi(n)=0.
	\end{equation*}
	Moreover,
	\begin{equation*}
		\|\varphi_{\G}\|_{\MoA(\G)}=|c_+|+|c_-|+\big(1-\tfrac{1}{q}\big)\|\big(I-\tfrac{\tau}{q}\big)^{-1}H\|_1,
	\end{equation*}
	where $\tau$ is the shift operator defined by~\eqref{tau}.
\end{maintheorem}

We now turn to the task of finding out which spherical functions on the Gelfand pair $(\G,\K)$ are completely bounded Fourier multipliers of $\G$, and find the explicit norms. According to~\cite{GV:HarmonicAnalysisOfSphericalFunctionsOnRealReductiveGroups} a continuous $\K$-bi-invariant function $\varphi_{\G}$ on $\G$ (which is not identically zero) is a spherical function on the Gelfand pair $(\G,\K)$ if and only if
\begin{equation*}
	f_{\G}\mapsto\int_{\G}f_{\G}(g)\varphi_{\G}(g)\dd\mu(g)
\end{equation*}
is multiplicative on the convolution algebra of compactly supported continuous $\K$-bi-invariant functions on $\G$, where $\mu$ is the Haar measure on $\G$ normalized such that $\mu(\K)=1$.
\begin{proposition}
	\label{sphfctcoinside}
	If $\varphi_{\G}$ and $\varphi_{\X}$ are related as in~\eqref{phipsi}, then $\varphi_{\G}$ is a spherical function on the Gelfand pair $(\G,\K)$ if and only if $\varphi_{\X}$ is a spherical function on $(\X,\Lambda_0)$.
\end{proposition}
\begin{proof}
	According to~\cite{Mau:SphericalFunctionsOverP-adicFields.I} the spherical functions on $(\G,\K)$ can be described in the following way:
	
	Let $\tau\in\padicnum$ be an element of order $1$, i.e., $|\tau|_q=\frac{1}{q}$. Define $y\in\G$ by
	\begin{equation*}
		y=\pi\left(\left(
		\begin{array}{cc}
			\tau & 0\\
			0 & 1
		\end{array}
		\right)\right),
	\end{equation*}
	where $\pi:\GLQ\to\PGLQ$ is the quotient map. Then
	\begin{equation*}
		\G=\bigsqcup_{n\in\N_0}\K y^n\K
	\end{equation*}
	and every spherical function $\Phi$ on $(\G,\K)$ is of the form $\Phi_z$ for a $z\in\C$, where
	\begin{equation}
		\label{UH*}
		\Phi_z(k_1y^n k_2)=\frac{q^{n(z-\tfrac{1}{2})}\big(q^{\tfrac{3}{2}+z}-q^{\tfrac{3}{2}-z}\big)-q^{-n(z-\tfrac{1}{2})}\big(q^{\tfrac{5}{2}-z}-q^{\tfrac{1}{2}+z}\big)}
		{(q+1)q^{\tfrac{n}{2}+1}\big(q^{z-\tfrac{1}{2}}-q^{\tfrac{1}{2}-z}\big)}
	\end{equation}
	for $k_1,k_2\in\K$ and $n\in\N_0$, cf.~\cite[(1.1), (2.2), (2.5) and~(8.2)]{Mau:SphericalFunctionsOverP-adicFields.I}. Note that if $\tau_1,\tau_2\in\padicnum$ satisfy $|\tau_1|_q=\frac{1}{q}$ and $|\tau_2|_q=\frac{1}{q}$ then $\tau_2=u\tau_1$ for some $u\in\padicints$. Therefore the definition of $\Phi_z$ does not depend on the choice of $\tau$. Since $|q|_q=\frac{1}{q}$ we can in the following put $\tau=q$.
	
	Let $(\varphi_z)_{z\in\C}$ be the spherical functions on the tree $\X=\G/\K$ (of degree $q+1$) given by~\eqref{varphiz2} and~\eqref{hz}. We claim that $\Phi_z$ and $\varphi_z$ are related as in~\eqref{phipsi}, i.e.,
	\begin{equation*}
		\Phi_z=\varphi_z\circ\rho\qquad(z\in\C)
	\end{equation*}
	where $\rho:\G\to\G/\K$ is the quotient map. Since $\Phi_z$ is $\K$-bi-invariant and $\varphi_z$ is radial, it is sufficient to check that
	\begin{equation*}
		\Phi_z(y^n)=\varphi_z(\rho(y^n))\qquad(n\in\N_0)
	\end{equation*}
	for $z\in\C$. Recall from Lemma~\ref{3X} that we can identify $\X=\G/\K$ with $\cL/\sim$, where the distance on $\cL/\sim$ is given by~\eqref{5.25}. Put $\Lambda_n=\rho(y^n)$ for $n\in\N_0$ considered as elements in $\cL/\sim$. Then $\Lambda_0=\rho(e)$ is the distinguished element in $\cL/\sim$ and $\Lambda_n=L_n/\sim$, where
	\begin{equation*}
		L_n=q^n\padicint\vec e_1+\padicint\vec e_2\qquad(n\in\N_0)
	\end{equation*}
	and where
	\begin{equation*}
		\vec e_1=\left(\begin{array}{c}1\\0\\\end{array}\right)\quad\mbox{and}\quad\vec e_2=\left(\begin{array}{c}0\\1\\\end{array}\right)
	\end{equation*}
	are the standard basis elements in $\padicnum^2$. In particular, $L_0=\padicint\vec e_1+\padicint\vec e_2$, so by~\eqref{5.2} and~\eqref{5.25}, $\dd(\Lambda_n,\Lambda_0)=n$. Thus, by~\eqref{varphiz2} and~\eqref{hz}
	\begin{equation}
		\label{UH**}
		\varphi_z(\rho(y^n))=f(z)q^{-n z}+f(1-z)q^{n (z-1)}\qquad(n\in\N_0),
	\end{equation}
	where
	\begin{equation*}
		f(z)=(q+1)^{-1}\frac{q^{1-z}-q^{z-1}}{q^{-z}-q^{z-1}}\qquad(z\in\C).
	\end{equation*}
	A simple computation shows that the right hand side of~\eqref{UH*} and~\eqref{UH**} coincide. Therefore,
	\begin{equation*}
		\Phi_z=\varphi_z\circ\rho\qquad(z\in\C),
	\end{equation*}
	which proves Proposition~\ref{sphfctcoinside}.
\end{proof}
Using Proposition~\ref{equalnorms} and~\ref{sphfctcoinside} and Theorem~\ref{sphSchurNormEigenvalue} we conclude the following.
\begin{theorem}
	\label{pgl2qqthm}
	Let $q$ be a prime number and consider the groups $\G=\PGLQ$ and $\K=\PSLZ$ and their quotient $\X=\G/\K$. Let $\varphi$ be a spherical function on the Gelfand pair $(\G,\K)$, then $\varphi$ is a completely bounded Fourier multiplier of $\G$ if and only if the eigenvalue $\eval$ of the corresponding spherical function on $\X$, is in the set
	\begin{equation*}
		\setw{\eval\in\C }{ \Re(\eval)^2+\big(\tfrac{q+1}{q-1}\big)^2\Im(\eval)^2<1 }\bigcup\seto{\pm1}.
	\end{equation*}
	The corresponding norm is given by
	\begin{equation*}
		\|\varphi\|_{\MoA(\G)}=\frac{|1-\eval^2|}{1-\Re(\eval)^2-\left(\frac{q+1}{q-1}\right)^2\Im(\eval)^2}\qquad(\Re(\eval)^2+\big(\tfrac{q+1}{q-1}\big)^2\Im(\eval)^2<1)
	\end{equation*}
	and
	\begin{equation*}
		\|\varphi\|_{\MoA(\G)}=1\qquad(\eval=\pm1).
	\end{equation*}
\end{theorem}
\begin{remark}
	It follows from~\eqref{sz}, and the proof of Proposition~\ref{sphfctcoinside}, that if $\varphi=\Phi_z$ (according to Mautner's parametrization~\eqref{UH*}), then $\eval$ in Theorem~\ref{pgl2qqthm} is given by $\eval_z=(1+\tfrac{1}{q})^{-1}(q^{-z}+q^{z-1})$.
\end{remark}
\begin{corollary}
	\label{MainThmCorollaryQ}
	Let $q$ be a prime number and consider the groups $\G=\PGLQ$ and $\K=\PSLZ$. There is no uniform bound on the ${\MoA(\G)}$-norm of the spherical functions on the Gelfand pair $(\G,\K)$ which are completely bounded Fourier multipliers.
\end{corollary}

	\bibliography{troelsBibliography}
	\contrib{Uffe Haagerup}{haagerup@imada.sdu.dk}\smallskip\\
Department of Mathematics and Computer Science, University of Southern Denmark, Campusvej~55, DK--5230~Odense~M, Denmark.\\

\bigskip\noindent
\contrib{Troels Steenstrup}{troelssj@imada.sdu.dk}\smallskip\\
Department of Mathematics and Computer Science, University of Southern Denmark, Campusvej~55, DK--5230~Odense~M, Denmark.\\

\bigskip\noindent
\contrib{Ryszard Szwarc}{szwarc2@gmail.com}\smallskip\\
Institute of Mathematics, University of Wroc{\l}aw, pl.\ Grunwaldzki~2/4, 50--384~Wroc{\l}aw, Poland.\smallskip\\\noindent and\smallskip\\ Institute of Mathematics and Computer Science, University of Opole, ul.\ Oleska~48, 45--052~Opole, Poland.\\

\end{document}